\documentclass[reqno,11 pt]{amsart}
\usepackage{amsmath,amssymb,latexsym}
\usepackage{hyperref}
\usepackage{amsthm}
\usepackage{color}
\usepackage[all,cmtip]{xy}
\usepackage{verbatim}

\hypersetup{
    colorlinks=true,
    linktoc=all,
        linkcolor=blue,
}

\newcommand{\Z}{{\mathbb{Z}}}
\newcommand{\Q}{{\mathbb{Q}}}

\newcommand{\OO}{{\mathcal{O}}}

\DeclareMathOperator{\tors}{tors}

\theoremstyle{plain}

\newtheorem{theorem}{Theorem}[section]
\newtheorem*{theorem*}{Theorem}

\newtheorem{proposition}[theorem]{Proposition}
\newtheorem{rem}[theorem]{Remark}

\newtheorem{lemma}[theorem]{Lemma}
\newtheorem{corollary}[theorem]{Corollary}
\newtheorem{defn}[theorem]{Definition}

\begin{document}

\title{Cyclic Cubic Extensions of $\Q$}
\maketitle

\begin{center}
Dipramit Majumdar \& B. Sury
\end{center}

\begin{abstract}
We determine the irreducible trinomials $X^3-aX+b$ for integers
$a,b$ which generate precisely all possible Galois extensions of
degree $3$ over $\Q$. The proof, although involved, is elementary
and one can parametrize all these polynomials explicitly. As an
accidental byproduct of the results, we prove that infinitely many
primes congruent to $1$ or $-1$ mod $9$ are sums of two rational
cubes - thereby, giving the first unconditional result on a
classical open problem.
\end{abstract}

\section{Introduction}

\noindent Let $K$ be a Galois extension of $\Q$ of degree $3$. Then
we can identify $K \cong \frac{\Q[X]}{(X^3-aX+b)}$ for some
irreducible polynomial $X^3-aX+b \in \Z[X]$ whose discriminant
$4a^3-27b^2$ is a perfect square.\footnote{Mathematics Subject
classification:
11D41; 12F05\\
Key words:  Cubic Galois extensions; Level-raising and lowering
maps; Diophantine equations; Primes as sums of two cubes}

\noindent {\it The aim of this article is to explicitly describe the
irreducible trinomials which give all the cubic Galois extensions of $\Q$.
This is done in the main theorem 4.6 at the end.} It is surprising that this seems to have been not done before.\\

\noindent A classical open problem asks for a classification of all cube free
natural numbers which can be expressed as sums of cubes of two
rational numbers. {\it As an accidental by-product of our main result, we prove that
infinitely primes congruent to $\pm 1$ modulo $9$ can be expressed
as a sum of two rational cubes.} Our proof seems to be the first unconditional one.\\

\noindent We call two polynomials $f(X)=X^3-a_1X+b_1 \in \Z[X]$ and
$g(X)=X^3-a_2X+b_2 \in \Z[X]$ to be {\bf equivalent} if there exists
a rational number $q \in \Q^\ast$ such that $a_2 = q^2a_1$ and
$b_2=q^3b_1$. Note that $Disc(g) = q^6 Disc(f)$. The aim of this
article is to find all irreducible polynomial $f(X)=X^3-aX+b \in
\Z[X]$ (up to equivalence) whose discriminant is a perfect square.
Since we are only interested in polynomials up to equivalence, we
need to find all the irreducible trinomials $f(X)=X^3-aX+b \in
\Z[X]$ each of which satisfies the following conditions.
\begin{enumerate}
\item There exists an integer $c$ such that $Disc(f)= 4a^3-27b^2=c^2 \neq 0$.
\item  $D= \mathrm{GCD}(a,b)$ is cube-free and, for every prime number $\ell$ such that $\ell^2 \mid D$, we have $\ell^3 \nmid b$.
\end{enumerate}

\noindent We briefly explain why the study of cubic Galois
extensions of $\Q$ reduces to the study of polynomials of the form
$X^3-aX+b$ satisfying the above two conditions. Note first that by a
linear change of variables, a cubic irreducible polynomial over $\Q$
can be taken to be of the form $X^3+uX+v$ for rational $u,v$ and a
further scaling by an integer $w$, where $w$ is a common denominator
for $u$ and $v$, generates the same field with a primitive element
whose minimal polynomial is of the form $X^3-aX+b$ for integers
$a,b$. Condition (1) arises as the Galois group of a cubic
polynomial is cyclic, of order $3$, if it is contained in the
alternating group $A_3$ - which happens if and only if the
discriminant is a perfect square (\cite{M}, Corollary 12.4).
Condition (2) arises because if $\ell$ is a prime such that $\ell^2
\mid GCD(a,b)$ and $\ell^3 \mid b$, then the polynomial $X^3-aX+b$
is equivalent to the polynomial $X^3-(a/\ell^2)X + (b/\ell^3)$.\\

\noindent  We write $(a,b)$ for the GCD of $a,b$. Let us note that
the determination of cubic trinomials $f(X)$ whose discriminant is a
perfect square reduces to integral solutions $( \sqrt{Disc(f)} , b ,
a)$ of $x^2 + 27 y^2 = 4z^3$.\\ Let $X_1$ denotes the affine curve
$x^2+27y^2=4z^3$ and let
$$X^D_1(\Z)= \{ (x,y,z) \in \Z^3 \mid
x^2+27y^2=4z^3 , xyz \neq 0 , (y,z)=D \}.$$ For a cube-free natural
number $D$, in view of property (2) above, we define
$$X_1^D(\Z)^* = \{(x,y,z) \in X_1^D(\Z): \ell^2|D \Rightarrow \ell^3 \not|
y~~\forall~~ {\rm prime}~~\ell  \}.$$ We observe that if $D$ is
square-free, then $X^D_1(\Z)=X^D_1(\Z)^*$. Note that $(x,y,z) \in
X^D_1(\Z)^*$ gives us a trinomial $X^3-zX+y$ which satisfies
conditions (1) and (2). As we vary $(x,y,z) \in X^D_1(\Z)^*$ for all
cube-free natural numbers $D$, the irreducible trinomials $X^3-zX+y$
give us all irreducible trinomials in $\Z[X]$ (up to equivalence)
whose discriminant is a perfect square. Thus, we start by
understanding the sets $X^D_1(\Z)^*$.\\
We say a cube-free natural number $D$ is {\bf admissible} if
$X^D_1(\Z)^*$ is non-empty. We will  show that $D$ is admissible if,
and only if, $D = D_1$ or $9D_1$ where either $D_1=1$ or each prime
factor
of $D_1$ is congruent to $1$ mod $3$.\\
So, we restrict our study to $X^D_1(\Z)^*$ for admissible $D$. \\
Observing that any solution $(x,y,z) \in X^D_1(\Z)$ gives us a
solution $(X,Y,Z)$ of $X^2 + 27Y^2=4DZ^3$ with $(Y,Z)=1$ (where
$x=DX, y=DY, z = DZ$), we let $X_D$ denote the affine curve $X^2 +
27Y^2=4DZ^3$ for any admissible $D$. We think $X_D$ as `level
curves' and define
$$X_D^1(\Z) =
\{(x,y,z) \in \Z^3: xyz \neq 0, x^2+27y^2=4Dz^3, (y,z)=1 \}.$$ Then
we have a bijection from $X^D_1(\Z) \to X^1_D(\Z)$ given by $(x,y,z)
\mapsto (\frac{x}{D}, \frac{y}{D}, \frac{z}{D})$. Moreover, this map
induces a bijection from $X^D_1(\Z)^*$ to $X^1_D(\Z)^*$, where
$$X_D^1(\Z)^* = \{(x,y,z) \in X_D^1(\Z): \ell^2|D \Rightarrow \ell \not|
y~~\forall~~ {\rm prime}~~\ell  \}.$$
We study the sets $X^1_D(\Z)^*$; note that $X^1_D(\Z)=X^1_D(\Z)^*$ if $D$ is square-free.  \\

\noindent  We observe that, for an admissible $D$ with $3 \nmid D$,
the solutions of $X^2 + 27Y^2=4DZ^3$ are related in a many-to-one
fashion with those of $X^2 + 3Y^2=4DZ^3$ - here, certain subtleties
arise as follows.\\
When $(x,y,z)$ is a solution to the first equation where $3 \mid z$, there are two solutions $(x,3y,z)$ and $(y, x/9, z/3)$ for the latter equation. \\
Also, if $(x,y,z)$ is a solution to the latter equation and $3 \mid y$, we get two solutions $(x,y/3,z)$ and $(9y,x,3z)$ of the former equation. \\
Let $Y_D$ denote the affine curve $X^2 + 3Y^2=4DZ^3$. Similar to the
case of the affine curve $X_D$, we define
$$Y_D^1(\Z) = \{(x,y,z) \in \Z^3: xyz \neq 0, x^2+3y^2=4Dz^3, (y,z)=1 \};$$
$$Y_D^1(\Z)^* = \{(x,y,z) \in Y_D^1(\Z): \ell^2|D \Rightarrow \ell \not| y~~\forall~~ {\rm prime}~~\ell  \}.$$
\vskip 3mm

\noindent  In Section 2, for admissible $D$ with $3 \nmid D$, we
obtain $X^1_D(\Z)^*$ in terms of $Y^1_D(\Z)^*$. Also, for an
admissible integer of the form $9D$, in section 2, we give a
bijection between $Y^1_D(\Z)$ and $X^1_{9D}(\Z)$. As a consequence,
this enables us to identify $X^1_{9D}(\Z)^*$ with a certain subset
of $Y^1_D(\Z)^*$. Therefore, we need to study $Y^1_D(\Z)^*$ for
admissible integers $D$
with $3 \nmid D$.\\

\noindent  In the study of $Y^1_D(\Z)^*$, the case $D=1$ is easy to
deal with. The solutions are obtained explicitly by reducing the
equation $X^2+3Y^2=4Z^3$ to the equation $x^2-3xy+3y^2=z^3$ that is
quickly solved using the arithmetic of the ring
$\mathbb{Z}[\omega]$, where $\omega$ is a primitive third root of
unity. A detailed proof of this can be found in
\cite{Cohen2}[Proposition 14.2.1(2)]. \vskip 3mm

\noindent  In Section 3, we construct $Y^1_D(\Z)^*$ from $Y^1_1(\Z)$
using maps between the integral points of curves of the form $X^2 +
3Y^2 = 4DZ^3$ for varying $D$ and keeping track of the GCDs of $Y$
and $Z$. These maps roughly ``trade off" the GCD of $(Y,Z)$ with a
coefficient of the $Z^3$ term. Informally, we call these fundamental
maps `level-raising' and `level-lowering'. As the book-keeping is
somewhat involved, it is convenient to define and study the maps
abstractly. \\

\noindent  Keeping track of the bookkeeping in Sections 2 and 3, we
determine in Section 4, all the irreducible trinomials (up
to equivalence) whose discriminant is a perfect square.

\noindent In Section 5, we relate integers expressible as sum of
rational cubes with integral solutions of $x^2+27y^2=4z^3$ and as a
consequence we prove that infinitely many primes congruent to $\pm 1$
modulo $9$ can be expressed as a sum of two rational cubes.\\

\noindent To summarize the flow of the paper, we make a few remarks.
The set of solutions of $x^2 + 27 y^2 = 4z^3$ with a given $(y,z)=D$
is connected naturally to the study of solutions of $x^2 + 3y^2 =
4z^3$ with the same $(y,z)=D$ (in a slightly subtle way - see Lemma
2.10). The main work consists of constructing the relevant part of
the set of solutions of $x^2+3y^2 = 4z^3$ with gcd$(y,z)=D$ (denoted
by $Y_D^1(\Z)^*$) from the set of solutions of $x^2+3y^2 =4z^3$ with
gcd$(y,z)=1$, using level-changing maps. It seems to us that if we
directly perform the various level-changing transformations within
the sets of solutions of $x^2+27y^2 = 4z^3$ with various $(y,z)$'s,
and try to bypass the equation $x^2 + 3y^2 = 4z^3$, it is
artificial, and we are not able to ensure that all points are
obtained. Hence, we are led to considering the level sets
$Y_D^1(\Z)^*$. \vskip 2mm

\noindent {\it The theorem 4.6 gives a complete list of
irreducible trinomials which generate all possible cubic Galois
extensions of $\Q$.} \vskip 3mm

\noindent  For convenience, we put down here a summary of notation
that will appear often in this paper. \vskip 2mm

\noindent $\bullet$ $X_1^D(\Z) = \{(x,y,z) \in \Z^3: xyz \neq 0,
x^2+27y^2=4z^3, (y,z)=D \}$.\\

\noindent $\bullet$ $X_1^D(\Z)^* = \{(x,y,z) \in X_1^D(\Z): \ell^2|D \Rightarrow \ell^3 \nmid y~~\forall~~ {\rm prime}~~\ell  \}$.\\

\noindent $\bullet$ $X_D^1(\Z) = \{(x,y,z) \in \Z^3: xyz \neq 0,
x^2+27y^2=4Dz^3, (y,z)=1 \}$.\\

\noindent $\bullet$ $X_D^1(\Z)^* = \{(x,y,z) \in X_D^1(\Z): \ell^2|D \Rightarrow \ell \nmid y~~\forall~~ {\rm prime}~~\ell  \}$.\\

\noindent $\bullet$ $Y_D^1(\Z) = \{(x,y,z) \in \Z^3: xyz \neq 0,
x^2+3y^2=4Dz^3, (y,z)=1 \}$.\\

\noindent $\bullet$ $Y_D^1(\Z)^* = \{(x,y,z) \in Y_D^1(\Z): \ell^2|D \Rightarrow \ell \nmid y~~\forall~~ {\rm prime}~~\ell  \}$.\\

\noindent The idea is to define `level-changing' maps between these
sets and determine the sets $X_1^D(\Z)^* = \{(x,y,z) \in X_1^D(\Z):
\ell^2|D \Rightarrow \ell^3 \nmid y~~\forall~~ {\rm prime}~~\ell \}$
from $Y_1^1(\Z) = \{(x,y,z) \in \Z^3: xyz \neq 0, x^2+3y^2=4z^3,
(y,z)=1 \}$; the latter set can be written down explicitly. \vskip
5mm

\section{Construction of $X^1_1(\Z)$ and $X^9_1(\Z)^*$}

Let $X_1$ be the affine curve $X^2+27Y^2=4Z^3$ as defined in the introduction. We define the set of trivial integral zeroes of $X_1$ by
$$X_1^{\text{triv}}(\Z) = \{ (x_0,y_0,z_0) \in X_1(\Z) | x_0y_0z_0 =0 \}=\{(0,\pm 2t^3, 3t^2), (\pm 2t^3,0,t^2) \mid t \in \Z\}.$$
Note that if $(x_0,y_0,z_0) \in X_1^{\text{triv}}(\Z)$, then $X^3-z_0
X + y_0$ is reducible. As we are interested in irreducible
trinomials (up to equivalence) with perfect square discriminant, we
study the sets $X_1^D(\Z)^* = \{(x,y,z) \in X_1^D(\Z): \ell^2|D
\Rightarrow \ell^3 \not| y~~\forall~~ {\rm prime}~~\ell  \}$ as $D$
vary over cube-free integers.\\

\noindent We first find all cube-free integers $D$ for which the set
$X^D_1(\Z)^*$ can possibly be non-empty. Towards that, we recall the
following elementary fact. We include a proof here.

\begin{lemma}\label{3k+1}
 Let $p \geq 5$ be an odd prime. Then there exists integer $u$ and $v$ such that  $p=u^2+3v^2$ if and
only if, $p \equiv  1 \pmod{3}$. Also,  $u$ and $v$ above are unique
up to sign, and $(u,3v)=1$. Moreover, any prime $p \equiv 1$ mod $3$
is expressible as $s^2-st+t^2$ for positive integers $s,t$ and the
expression $p=u^2+3v^2$ has the unique positive solution $u = s +
\frac{t}{2}, v = \frac{t}{2}$ when $s$ is odd and $t$ is even and $u
= \frac{s+t}{2}, v = \frac{|s-t|}{2}$ when $s,t$ are odd. Further,
if an integer $N$ is expressible in the form $a^2+3b^2$ with
$(a,3b)=1$, then its odd prime factors are of the form $3k+1$.
\end{lemma}

\begin{proof}
Clearly, if $p \neq 3$ is of the form $x^2+3y^2$, then it is $1$ mod
$3$. Conversely, let $p \equiv 1$ mod $3$ be a prime. As $3$ divides
$|\mathbf{F}_p^{\ast}|$ there exists an element of order $3$ in the
cyclic group $\mathbf{F}_p^{\ast}$. That is, there exists an integer
$a \not\equiv 1$ mod $p$ but $a^3 \equiv 1$ mod $p$. Thus, $p$
divides $a^2+a+1$ so that $p$ divides $|-a+\omega|^2$ where $\omega$
is a primitive cube root of unity. Clearly, $p$ is not irreducible
(as it is not prime) in the unique factorization domain
$\mathbf{Z}[\omega]$. Hence $p = (a+b \omega)(c+d \omega)$ with each
factor a non-unit, which gives $p = |a+b \omega|^2 = a^2-ab+b^2$.
Therefore, either $a,b$ are both odd or one of them (say $a$) is odd
and the other even, In the first case, $p =
((a+b)/2)^2+3((a-b)/2)^2$ and, in the 2nd case, $p =
(a-b/2)^2+3(b/2)^2$. This completes the proof of the first statement.
The fact that $(u,3v)=1$ is obvious. \\
To prove uniqueness (up to sign), let $p=u_1^2+3v_1^2 = u_2^2+
3v_2^2$. Now,
$$(u_1v_2 -v_1u_2)(u_1v_2 + v_1u_2) = u_1^2v_2^2 - v_1^2u_2^2 \equiv -v_1^2(u_2^2 + 3v_2^2) \equiv 0 \pmod{p}.$$
Thus, either $p \mid (u_1v_2 -v_1u_2)$ or $p \mid (u_1v_2 + v_1u_2)$.\\
If $p \mid (u_1v_2 -v_1u_2)$, then $p \mid (u_1(u_1v_2 -v_1u_2) -
pv_2)$; that is, $p \mid -v_1(u_1u_2+3v_1v_2)$. Hence $p \mid
(u_1u_2 + 3v_1v_2)$ as $(p,v_1)=1$. Since
$$1 = \big( \frac{u_1u_2 + 3v_1v_2}{p} \big)^2 + 3 \big( \frac{u_1v_2 - v_1u_2}{p} \big)^2,$$ we get
$\frac{u_1v_2 - v_1u_2}{p}=0$ and $\frac{u_1u_2 + 3v_1v_2}{p} = \pm
1$. Thus, $u_1v_2 = u_2v_1$ and $u_1u_2 + 3v_1v_2 = \pm{p}$.
So, $\pm{p}u_1 = u_1^2u_2 + 3v_1u_1v_2= (u_1^2+3v_1^2)u_2 = pu_2$.\\
We conclude that $u_1 = \pm u_2$ and $v_1 = \pm v_2$. \\
The case $p \mid (u_1v_2 +v_1u_2)$ is similar. This concludes the
proof of uniqueness of $u,v$ in the
expression of a prime $p = u^2+3v^2$ (up to sign).\\
Finally, let $N=a^2+3b^2$ with $(a,3b)=1$. Let $p$ be an odd prime
such that $p \mid N$. Since $(a,3b)=1$, we have $p \neq 3$ and $p
\nmid ab$. since $a^2+3b^2 \equiv 0 \pmod{p}$, we see that $-3
\equiv \big( \frac{a}{b} \big)^2 \pmod{p}$, that is $\big(
\frac{-3}{p} \big) =1$. From the quadratic reciprocity law, it
follows that $p \equiv 1 \pmod{3}$.
\end{proof}

\begin{proposition}\label{improperD}
Let $D$ be a cube free integer. Consider the set $X_1^D(\Z)^* =
\{(x,y,z) \in X_1^D(\Z): \ell^2|D \Rightarrow \ell^3 \nmid
y~~\forall~~ {\rm prime}~~\ell  \}$ where $X_1^D(\Z) = \{(x,y,z) \in
\Z^3: xyz \neq 0, x^2+27y^2=4z^3, (y,z)=D \}$. If a prime $\ell
\equiv 2 \pmod{3}$ divides $D$, then $X^D_1(\Z)^*$ is empty.
\end{proposition}

\begin{proof}
Let $(x_0,y_0,z_0) \in X^D_1(\Z)^*$. We divide the proof in two cases.

\underline{{\bf Case I: $\ell =2.$}}

Write $y_0 = 2 y_1, z_0 = 2 z_1$, this implies $2 \mid x_0$. Write $x_0 = 2 x_1$. Simplifying we get,
$$x_1^2+ 27y_1^2 = 8 z_1^3.$$
Thus $x_1^2 + 27 y_1^2 \equiv 0 \pmod{8}$. For any integer $a$, we have $a^2 \equiv 0, 1, 4 \pmod{8}$. Thus this implies either $x_1 \equiv y_1 \equiv 0 \pmod{4}$ or  $x_1 \equiv y_1 \equiv 2 \pmod{4}$.

In the first case, $4 \mid x_1$ and $4 \mid y_1$ implies $2 \mid
z_1$. Thus, $4 \mid (y_0,z_0)=D$ and $8 \mid y_0$ which implies
$(x_0,y_0,z_0) \notin X^D_1(\Z)^*$.

In the second case, writing $x_1 = 2(2r+1)$ and $y_1= 2(2s+1)$, we obtain
$$(2r+1)^2 + 27(2s+1)^2 = 2 z_1^3.$$
Since the LHS is divisible by $4$, this implies $z_1$ is even. Writing $z_1 =2t$ we obtain
$$(2r+1)^2 + 27(2s+1)^2 = 16 t^3,$$
which is not possible, as RHS $ \equiv  0 \pmod{8}$ but LHS $\equiv
4 \pmod{8}$. \vskip 2mm

\underline{{\bf Case II: $2< \ell \equiv 2\pmod{3}.$}}

The proof of Lemma \ref{3k+1} implies that $x^2+3y^2=0$ has no
solution other than $(0,0) \in \mathbb{F}_{\ell}^2$.

Suppose that $X^2+27Y^2=4Z^3$ has a non-trivial solution
$(x_0,y_0,z_0)$ with GCD $(y_0,z_0)=D$. If $\ell \mid y_0$ and $\ell
\mid z_0$, then $\ell \mid x_0$. Writing $x_0 = \ell x_1, y_0 = \ell
y_1, z_0= \ell z_1$ and simplifying we see
$$ x_1^2 + 27y_1^2 = 4\ell z_1^3 \text{ hence } x_1^2 + 3 (3y_1)^2 =0 \pmod{\ell}.$$
Hence $(x_1, 3y_1) = (0,0) \in \mathbb{F}_{\ell}^2$, that is $\ell \mid x_1$ and $\ell \mid y_1$. This implies $\ell^2 \mid D$ and $\ell \mid z_1$. Writing $x_1 = \ell x_2, y_1 = \ell y_2, z_1 = \ell z_2$ and simplifying, we see that
$$ x_2^2 + 27y_2^2 = 4\ell^2 z_2^3 \text{ hence } x_2^2 + 3 (3y_2)^2 =0 \pmod{\ell}.$$
Again by a similar argument we see that $\ell \mid y_2$, hence $(x_0, y_0,z_0) \notin X^D_1(\Z)^*$.
\end{proof}

\begin{proposition}
Let $D$ be a cube free integer. If $3 \mid \mid D$, then $X^D_1(\Z)^*$ is empty.
\end{proposition}

\begin{proof}
Suppose $(x_0,y_0,z_0) \in X^D_1(\Z)^*$. We see that $ 3 \mid z_0$ implies $27 \mid x_0^2$, hence $9 \mid x_0$. Now $9 \mid x_0$ and $3 \mid y_0$ implies $81 \mid z_1^3$, hence $9 \mid z_0$. Next we see that $ 3 \mid y_0$ and $ 9 \mid z_0$ implies $3^5 \mid x_0^2$, hence $3^3 \mid x_0$. Lastly notice that $3^3 \mid x_0$ and $9 \mid z_0$  implies $ 3^6 \mid 27y_0^2$, hence $9 \mid y_0$. Thus we see that $9$ divides both $y_0$ and $z_0$, hence $9 \mid D$, which contradicts the fact that $ 3 \mid \mid  D$.
\end{proof}
\vskip 2mm

\noindent In view of the above two propositions when our set $X^D_1(\Z)^*$ is
empty, it is meaningful to consider the following integers $D$ only.
\vskip 2mm

\begin{defn}
(i) We call a natural number $D$ {\bf admissible} if it is cube free
and of the form $D_1$ or $9D_1$, where $D_1=1$ or all the prime
factors of $D_1$ are congruent to 1 modulo $3$. \\
(ii) For admissible $D$, we consider the curve $X_D:
X^2+27Y^2=4DZ^3$. As mentioned in the introduction, let
$$X_D^1(\Z) =
\{(x,y,z) \in \Z^3: xyz \neq 0, x^2+27y^2=4Dz^3, (y,z)=1 \};$$
$$X_D^1(\Z)^* = \{(x,y,z) \in X_D^1(\Z): \ell^2|D \Rightarrow \ell \not| y~~\forall~~ {\rm prime}~~\ell  \}.$$
\end{defn}
\vskip 2mm

\noindent We observe now that the sets $X_D^1(\Z)$ and $X_D^1(\Z)^*$
are in natural bijection, respectively, with the sets $X_1^D(\Z)$
and $X_1^D(\Z)^*$. \vskip 2mm

\begin{lemma}\label{thetaD}
The map
  $\theta_D : X_D^1(\Z) \to X_1^D(\Z)$
given by $(x,y,z) \mapsto (Dx,Dy,Dz)$ is a bijection and restricts
to a bijection from $X_D^1(\Z)^*$ to $X_1^D(\Z)^*$.
\end{lemma}

\begin{proof}
Indeed, note that if $(x,y,z) \in X^1_D(\Z)$, then
$(Dx)^2+27(Dy)^2=D^2(x^2+27y^2)=D^2(4Dz^3)=4(Dz)^3$,
hence $(Dx,Dy,Dz) \in X_1(\Z)$. Also since $(y,z)=1$, this implies $(Dy,Dz)=D$, hence $(Dx,Dy,Dz) \in X^D_1(\Z)$. Also note that, if $D$ is an admissible integer, then $(x,y,z) \in X^1_D(\Z)^*$ iff $(Dx,Dy,Dz) \in X_1^D(\Z)^*$.\\
Conversely, suppose $(x',y',z') \in X_1^D(\Z)$. Then $(y',z')=D$ and
hence $D$ divides $x'$, $y'$ and $z'$. Define the map
$$\theta_D^{-1}: X_1^D(\Z) \to X_D^1(\Z); (x',y',z') \mapsto
(\frac{x'}{D}, \frac{y'}{D}, \frac{z'}{D}).$$ This is evidently the
map $\theta_D^{-1}$ and maps $X_1^D(\Z)^*$ to $X_D^1(\Z)^*$.\\
We conclude that $\theta_D$ defines a bijection from $X_D^1(\Z)$ to
$X_1^D(\Z)$ and restricts to a bijection from $X_D^1(\Z)^*$ onto
$X_1^D(\Z)^*$.
\end{proof}
\vskip 3mm

\subsection{A Related Curve}

\noindent For an admissible integer $D$ such that $3 \nmid D$, we
consider the level curve $Y_D: X^2+3Y^2 = 4DZ^3$. As before, we
define
$$Y_D^1(\Z) = \{(x,y,z) \in \Z^3: xyz \neq 0,
x^2+3y^2=4Dz^3, (y,z)=1 \};$$
$$Y_D^1(\Z)^* = \{(x,y,z) \in Y_D^1(\Z): \ell^2|D \Rightarrow \ell \not| y~~\forall~~ {\rm prime}~~\ell  \}.$$

\begin{lemma}
Let $D$ be an admissible integer with $3 \nmid D$. Then the map
$$\delta_D: Y^1_D(\Z) \to X^1_{9D}(\Z) \text{ given by } \delta(x,y,z)= (3x,y,z)$$
defines a bijection. We remark that the assumption $3 \nmid D$ ensures that $9D$ is admissible.
\end{lemma}

\begin{proof}
A simple calculation shows that $(x,y,z) \in Y^1_D(\Z)$ implies $(3x,y,z) \in X^1_{9D}(\Z)$.\\
Conversely, suppose that $(x,y,z) \in X^1_{9D}(\Z)$. Then $x^2+
27y^2 = 36Dz^3$, hence $3 \mid x$. Again an easy calculation shows
$(\frac{x}{3}, y, z) \in Y^1_D(\Z)$.
\end{proof}

\begin{corollary}\label{9D}
Let $D$  be an admissible integer with $3 \nmid D$. Then the map
$\delta_D$ above defines a bijection of
$$S^\prime_D=\{ (x,y,z) \in
Y_D^1(\Z)^* \mid 3 \nmid y \}$$ with $X^1_{9D}(\Z)^*$.
\end{corollary}

\begin{proof}
If $(x,y,z) \in S^\prime_D$, then $\delta_D(x,y,z)=(3x,y,z) \in
 X^1_{9D}(\Z)$. It is obvious that for a prime $\ell \equiv 1 \pmod{3}$, $\ell^2
\mid D \iff \ell^2 \mid 9D$ and in this situation $\ell \nmid y$ as
$(x,y,z) \in Y_D^1(\Z)^*$.
Moreover, $3 \nmid y$ as $(x,y,z) \in S^\prime_D$. Hence $\delta_D(x,y,z) \in X^1_{9D}(\Z)^*$.\\
Conversely, if $(a,b,c) \in X^1_{9D}(\Z)^* \subset X^1_{9D}(\Z)$,
then for a prime $\ell \equiv 1 \pmod{3}$, $\ell^2 \mid D \iff
\ell^2 \mid 9D$ and in this situation $\ell \nmid b$ as $(a,b,c) \in
X_{9D}^1(\Z)^*$. Moreover, we have $3 \nmid b$. We conclude that
$(\frac{a}{3},b,c) \in S^\prime_D$.
\end{proof}

\noindent We recall the easy parametrization of the set $Y^1_1(\Z)$
as mentioned in the introduction.

\begin{theorem}\cite{Cohen2}[Proposition 14.2.1(2)]\label{primitivesol}
The equation $X^2+3Y^2=4Z^3$ in non-zero integers $x,y$ and $z$ with
$x$ and $y$ coprime has two disjoint parametrization
$$(x,y,z)= ((s+t)(2s-t)(s-2t), 3st(s-t), s^2-st+t^2)$$
$$(x,y,z)=(\pm((s+t)^3-9st^2), s^3-3s^2t+t^3, s^2-st+t^2)$$
where in both cases $s$ and $t$ are co-prime integers with $3 \nmid
(s+t)$. The first parametrization corresponds to the case $6 \mid y$
and the second where $6$ is coprime to $y$.
\end{theorem}

\noindent Using this we obtain the following description of $X^9_1(\Z)^*$ as follows.

\begin{theorem}\label{X91}
$$X_1^9(\Z)^* = \{ \pm 27 ((s+t)^3-9st^2), 9(s^3-3s^2t+t^3), 9( s^2-st+t^2) \}$$
where $s$ and $t$ are co-prime integers with $3 \nmid (s+t)$.
\end{theorem}

\begin{proof}
Composition of the maps $\delta_9$ and $\theta_9$ gives us a bijection
$$\theta_9 \circ \delta_9: S^\prime_1 \to X^1_9(\Z)^* \to X^9_1(\Z)^* \text{ given by } (x,y,z) \mapsto (27x,9y,9z).$$
Since $S^\prime_1 = \{ (\pm((s+t)^3-9st^2), s^3-3s^2t+t^3, s^2-st+t^2) \mid (s,t)=1, 3 \nmid s+t \}$ the result follows.
\end{proof}

\noindent We are interested in determining the set $X_D^1(\Z)^* =
\{(x,y,z) \in \Z^3: xyz \neq 0, x^2+27y^2=4Dz^3, (y,z)=1, \ell^2|D
\Rightarrow \ell \not| y~~\forall~~ {\rm prime}~~\ell  \}$ in terms
of the set $Y_D^1(\Z)^* = \{(x,y,z) \in \Z^3): xyz \neq 0, x^2+3y^2
= 4Dz^3, (y,z)=1, \ell^2|D \Rightarrow \ell \not| y~~\forall~~ {\rm
prime}~~\ell  \}$. \vskip 2mm

\begin{lemma}\label{D}
Let $D$ be an admissible integer with $3 \nmid D$. Then
$$X^1_D(\Z)^* =  \{ (9y, x, 3z) \mid (x,y,z) \in Y^1_D(\Z)^* \} \cup \{ (x, \frac{y}{3} , z) \mid (x,y, z) \in Y^1_D(\Z)^*, 3 \mid y \}.$$
\end{lemma}

\begin{proof} Let us write $X^1_D(\Z)^*$ as a disjoint union of the sets $T_D$ and $T_D'$, where
$$T_D = \{ (x,y,z) \in X^1_D(\Z)^* \mid 3 \nmid z\}~,~T_D' =  \{ (x,y,z) \in X^1_D(\Z)^* \mid 3 \mid z\}.$$
First, we show there is a bijection between $T_D$ and $S_D=  \{ (x,y,z) \in Y^1_D(\Z)^* \mid 3 \mid y\}$.
We have a natural map $\beta_D : T_D \to S_D$ given by $\beta_D(x,y,z) = (x, 3y, z)$.
Note that it is essential that $3 \nmid D$ for the map $\beta_D$ to make sense
(if $ 3 \mid D$, then $9 \mid D$, which will imply that the set $S_D$ is empty).
We also have a natural map $\alpha_D : S_D \to T_D$ given by $\alpha_D(x,y,z)= (x, \frac{y}{3}, z)$.
Observe that $\alpha_D \circ \beta_D = id_{T_D}$ and $\beta_D \circ \alpha_D = id_{S_D}$.
As a consequence, we can identify $T_D$ with $\alpha_D(S_D) = \{ (x, \frac{y}{3} , z) \mid (x,y, z) \in Y^1_D(\Z)^*, 3 \mid y \}$.\\
Now we will show that there is a bijection between $T'_D$ and $Y^1_D(\Z)^*$ as follows.\\
We define a map $\beta'_D: T'_D \to Y^1_D(\Z)^*$ given by
$\beta'_D(x,y,z) = (y, \frac{x}{9}, \frac{z}{3})$. To ensure that
the map is well-defined, we need to check:\\ $(i)~ 9 \mid x, (ii)~
GCD (\frac{x}{9}, \frac{z}{3}) =1$ and $(iii)$ if $\ell^2 \mid D$,
then $\ell \nmid \frac{x}{9}$. \\Now if $(x,y,z) \in X^1_D(\Z)$ and
$3 \mid z$, then it follows that $9 \mid x$. Moreover in this
situation, it follows that if  $GCD (\frac{x}{9} ,
\frac{z}{3})=d>1$, then $d \mid y$, which contradicts $GCD (y,z)=1$.
Further, as $(x,y,z) \in X^1_D(\Z)^*$, it follows that if $\ell^2
\mid D$, then $\ell \nmid y$. Hence $\ell \nmid y^2 =
(4D(\frac{z}{3})^3 - 3(\frac{x}{9})^2)$, which implies $\ell \nmid
\frac{x}{9}$ (as $\ell \neq 3$).
This shows that the map $\beta'_D$ is well defined.\\
We also have a map $\alpha'_D: Y^1_D(\Z)^* \to T'_D$ given by
$\alpha'_D(x,y,z) = (9y, x, 3z)$. To ensure that the map is well
defined, we need to check:\\ $(i)~  GCD (x,3z) =1$ and $(ii)$ if
$\ell^2 \mid D$, then $\ell \nmid x$. \\Note that if $(x,y,z) \in
Y^1_D(\Z)^*$ with $3 \mid x$, then $3 \mid z$ (as $3 \nmid D$),
which in turn shows that $3 \mid y$, which contradicts $GCD
(y,z)=1$. This shows that $GCD (x,3z) = GCD (x,z)= d$ with $ 3\nmid
d$. Now if $GCD (x,z) = d >1$, then $d^2 \mid 3y^2 = 4Dz^3- x^2$ and
hence $d \mid y$, which contradicts $GCD (y,z)=1$. Thus $GCD (x,3z)
= 1$. Moreover, as $(x,y,z) \in Y^1_D(\Z)^*$, it follows that if
$\ell^2 \mid D$, then $\ell \nmid y$. Hence $\ell$ does not divide
$x^2 = (4Dz^3-27y^2)$ as $\ell^2$ divides $D$ and $\ell \neq 3$.
This shows that the map $\alpha'_D$ is well defined.\\
Now observe that $\alpha'_D \circ \beta'_D = id_{T'_D}$ and $\beta'_D \circ \alpha'_D = id_{Y^1_D(\Z)^*}$. \\
As a consequence, we can identify $T'_D$ with
$\alpha'_D(Y^1_D(\Z)^*) =  \{ (9y, x, 3z) \mid (x,y,z) \in
Y^1_D(\Z)^* \}$. This completes the proof of the lemma.
\end{proof}

\noindent Using the parametrization of $Y^1_1(\Z)$ and Lemma \ref{D}
above, we obtain immediately:

\begin{theorem}\label{X11}
The set $X_1^1(\Z) = \{(x,y,z) \in \Z^3: xyz \neq 0, x^2+27y^2=4z^3,
(y,z)=1\}$
 is given by the parametrizations
\begin{multline*}
 \{ (9(s^3-3s^2t+t^3), \pm((s+t)^3-9st^2), 3(s^2-st+t^2) ) \} \cup  \\
 \{ (27st(s-t), (s+t)(2s-t)(s-2t), 3(s^2-st+t^2) ) \} \cup \{ ((s+t)(2s-t)(s-2t), st(s-t), s^2-st+t^2)    \},
 \end{multline*}
 where $s$ and $t$ are coprime integers with $3 \nmid (s+t)$.
\end{theorem}

 \noindent {\bf Remark.} We remark that the three sets that occur in the parametrization appearing in Theorem \ref{X11} are disjoint.
To see this, note that \\$(s+t)(2s-t)(s-2t) = (s+t)(3s -
(s+t))((s+t) -3t) \equiv -(s+t)^3 \pmod{3}$ and\\ $s^3-3s^2t+t^3 =
(s+t)^3 - 3(2s^2t+ st^2 \equiv (s+t)^3 \pmod{3}$.\\ Thus the
condition $3 \nmid (s+t)$ implies $9 \mid \mid 9(s^3-3s^2t+t^3)$ and
$3 \nmid (s+t)(2s-t)(s-2t)$. This is equivalent to asserting that
the first co-ordinate of an element is exactly divisible by $9$ in
the first parametric set, is divisible by $27$ in the second
parametric set, and is coprime to $3$ in the third parametric set.

\section{Integral points on the curve $X^2+3Y^2=4DZ^3$.}

\noindent Recall that $Y_D$ denotes the affine curve $X^2+3Y^2=4DZ^3$. By $Y_D(\Z)$ we denote the set of integral points on the affine curve $Y_D$, that is $Y_D(\Z) = \{ (a,b,c) \in \Z^3 \mid a^2 + 3b^2 = 4Dc^3 \}$. In this section, for a cube-free natural number $D$, we describe the set
$$Y^1_D(\Z) =  \{ (X,Y,Z) \in \Z^3 \mid X^2+3Y^2=4DZ^3 , XYZ \neq 0 , (Y,Z)=1 \}.$$
We observe that if for a prime $\ell$ , $\ell^2 \mid D$ and $(x,y,z)
\in Y^1_D(\Z)$ with $ \ell \mid y$,
then $\ell \mid x$ and we have $(\frac{x}{\ell}, \frac{y}{\ell}, z) \in Y^1_{\frac{D}{\ell^2}}(\Z)$. \\

\noindent Using this, we can inductively describe $Y^1_D(\Z)$ as follows:\\
Assume that $3 \nmid D$. Write $D=\ell_1^2  \cdots \ell_r^2 D_1$, where $D_1$ is square-free with $\ell_1\cdots \ell_r \nmid D_1$.\\
We have $Y^1_{D_1}(\Z) = Y^1_{D_1}(\Z)^*$.\\
Now
\begin{align*}
Y^1_{\ell_1^2D_1}(\Z) & = \{ (x,y,z) \in Y^1_{\ell_1^2D_1}(\Z) \mid \ell_1 \nmid y \} \sqcup \{ (x,y,z) \in Y^1_{\ell_1^2D_1}(\Z) \mid \ell_1 \mid y \} \\
& = Y^1_{\ell_1^2D_1}(\Z)^* \sqcup \{ (\ell_1 x_1, \ell_1 y_1, z) \mid (x_1, y_1,z) \in  Y^1_{D_1}(\Z)^*, \ell_1 \nmid z \}.
\end{align*}
For the last equality, we have used the fact that $\ell_1 \mid y$
implies $\ell_1 \mid x$ and $(\frac{x}{\ell_1}, \frac{y}{\ell_1}, z)
\in Y^1_{D_1}(\Z)= Y^1_{D_1}(\Z)^*$.

\noindent Proceeding in this manner, let $1 <B_1\neq B_2  \neq
\cdots \neq B_{2^r-1} \le \ell_1 \cdots \ell_r$ be the set of
divisors of $\ell_1\cdots \ell_r$. We have
$$Y^1_D(\Z) = Y^1_D(\Z)^* \sqcup_{i=1}^{2^r-1} \{ (B_i x, B_i y, z) \mid (x,y,z) \in Y^1_{\frac{D}{B_i^2}}(\Z)^*, B_i \nmid z \}.$$

\noindent Thus, for any cube-free integer $D$, with $3 \nmid D$, we can explicitly describe the set $Y^1_D(\Z)$ if we know the set $Y^1_D(\Z)^*$
for all cube-free integers $D$, with $3 \nmid D$.\\
We observe that for any cube-free integer $D$, with $3 \nmid D$,
$(a,b,c) \in Y^1_D(\Z)$ implies $3 \nmid c$. We can construct
$Y^1_{3D}(\Z)$ and $Y^1_{9D}(\Z)$ from $Y^1_D(\Z)$ as follows:
$$Y^1_{3D}(\Z) = \{ (3b, a, c) \mid (a,b,c) \in Y^1_D(\Z)\}  \text{ and }  Y^1_{9D}(\Z) = \{ (3a, 3b, c) \mid (a,b,c) \in Y^1_D(\Z)\}.$$
Thus, it suffices to understand the sets $Y^1_D(\Z)^*$ for all cube-free $D$, with $3 \nmid D$.\\

\begin{proposition}
Let $D$ be a cube-free integer, with $3 \nmid D$. If a prime $\ell \equiv 2 \pmod{3}$ divides $D$, then $Y^1_D(\Z)^*$ is empty.
\end{proposition}

\begin{proof}
The proof is similar to that of Proposition \ref{improperD}; hence it is omitted.
\end{proof}

\noindent In view of the above discussion, we see that it is enough
to understand the sets $Y^1_D(\Z)^*$ for {\bf admissible} integers
$D$ for which $3 \nmid D$. In the remaining part of the section, we
describe how to obtain $Y^1_D(\Z)^*$ from $Y^1_1(\Z)$ using level
raising and level lowering maps. We first describe the set
$Y^1_p(\Z)$ for a prime $p$ of the form $3k+1$. The case for general
admissible  $D$ for which $3 \nmid D$ is similar, but more
notationally involved.

\begin{defn}[Level Raising By A Prime $p \equiv 1 \pmod{3}$]
Let $p$ be a prime which is congruent to $1$ modulo $3$. Then we can write $p$ uniquely as $p=u^2+3v^2$ and $(u,3v)=1$ with $u, v>0$. We define two level raising maps by $p$ as follows.
 $$[p]^{+} : Y^1_D(\Z) \to Y_{pD}(\Z) \text{ and } [p]^{-} : Y^1_D(\Z) \to Y_{pD}(\Z)$$
 given by
 $$[p]^{+}(x,y,z)=(ux + 3vy, uy - vx, z) \text{ and } [p]^{-}(x,y,z)=(ux - 3vy, uy + vx, z).$$
\end{defn}
\noindent The definition is justified by the following calculation.
$$(ux \pm 3vy)^2 + 3 (uy \mp vx)^2 = (x^2+3y^2)(u^2+3v^2) = 4pDz^3.$$

\begin{defn}[Level Lowering By A Prime $p \equiv 1 \pmod{3}$]\label{defnlevelloweing}
Let $p$ be a prime which is congruent to $1$ modulo $3$. Then we can write $p$ uniquely as $p=u^2+3v^2$ and $(u,3v)=1$ with $u, v>0$. We  define the level lowering map by $p$ as follows
\begin{equation*}
[p]_{\ast} : Y^1_{pD}(\Z)^* \to Y_{D}(\Z)
\end{equation*}
 $$[p]_{\ast}(x,y,z) \mapsto
 \begin{cases}
 \big(\frac{ux - 3vy}{p}, \frac{uy + vx}{p}, z \big) & \text{ if } p \mid (uy+vx), \\
 \big(\frac{ux + 3vy}{p}, \frac{uy - vx}{p}, z \big)  & \text{ if } p \nmid (uy+vx).
 \end{cases}
 $$

\end{defn}

\noindent The definition is justified as follows.\\

\noindent  If $(x,y,z) \in Y^1_{pD}(\Z)$, then
 $$u^2y^2-v^2x^2 \equiv -3v^2y^2 -v^2x^2 =-v^2(x^2+3y^2)=-4pDv^2z^3 \equiv 0 \pmod{p}.$$
 Thus $p$ divides either $(uy+vx)$ or $(uy-vx)$.
 Note that if $(x,y,z) \in Y^1_{pD}(\Z)^*$ and $p$ divides both of $uy+vx$ and $uy-vx$, then it implies that $p$ divides both $x$ and $y$.
Hence if $p \mid D$, then $p^2 \mid pD$ and $p \mid y$ so that
$(x,y,z) \notin Y^1_{pD}(\Z)^*$.
 On the other hand, if $p$ divides both $x$ and $y$ as above, and $p \nmid D$, then $p \mid z$. Hence $(x,y,z) \notin Y^1_{pD}(\Z)$. \\
 So, we conclude that $(x,y,z) \in Y_{pD}^1(\Z)^*$ implies $p$ divides exactly one of $\{ uy+vx, uy-vx \}$. If $p \mid (uy \pm vx)$, then $p$ divides
$u(uy\pm vx)-py = \pm v(ux \mp 3vy)$; hence $p \mid (ux \mp 3vy)$.\\
Finally, note that
$$\Big(\frac{ux \mp 3vy}{p} \Big)^2 + 3 \Big(\frac{uy \pm vx}{p}\Big)^2 = \frac{1}{p^2}(u^2+3v^2)(x^2+3y^2) = 4Dz^3.$$

\begin{lemma}
Let $D$ be an admissible integer and $p \mid D$ be a prime (which is necessarily congruent to $1$ modulo $3$). Then, the level
lowering map $[p]_\ast$ as defined in Definition \ref{defnlevelloweing} maps $Y^1_D(\Z)^*$ to
$Y^1_{\frac{D}{p}}(\Z)^*$.
\end{lemma}

\begin{proof}
Suppose $(x_0,y_0,z_0) \in Y^1_D(\Z)^*$ and let
$[p]_\ast(x_0,y_0,z_0)=(\tilde{x}, \tilde{y},z_0)$. Suppose that $(\tilde{y},z_0)=d$, then
$(\tilde{x},z_0)=d$, hence $ d \mid \tilde{x}$ and $ d \mid
\tilde{y}$. Thus $d \mid (u \tilde{y} \mp v \tilde{x})$; in
particular, $d \mid y_0$. Thus $d \mid (y_0,z_0)=1$. So,
$[p]_\ast(x_0,y_0,z_0) \in Y^1_{\frac{D}{p}}(\Z)$. Note that $p^2
\nmid \frac{D}{p}$ as $D$ is cube-free. Let $\ell_1 \neq p$ be a
prime such that $\ell_1^2 \mid D$ and $(x_0,y_0,z_0) \in
Y^1_D(\Z)^*$, then $\ell_1 \nmid y_0$. Now, notice that,
$$u^2y_0^2-v^2x_0^2 = py_0^2 -v^2(x_0^2+3y_0^2)=py_0^2-4Dv^2z_0^3 \equiv py_0^2 \neq 0 \pmod{\ell_1}.$$
Note that $\tilde{y}$ is either $(uy_0+vx_0)/p$ or $(uy_0-vx_0)/p$.
Thus $\ell_1 \nmid \tilde{y}$. Hence $(\tilde{x}, \tilde{y},z_0) \in
Y^1_{\frac{D}{p}}(\Z)^*$.
\end{proof}

\noindent
Our aim is to construct $Y^1_D(\Z)^*$  from $Y^1_1(\Z)$ using level
raising maps. For each prime divisors $p$ of $D$, we
apply $p$-level raising maps (twice if $p^2|D$). The main issue is to prove that we have constructed the
whole $Y^1_D(\Z)^*$. For this, we need to use the level lowering
maps. Now we study the image of the level lowering and level
raising map.\\

\noindent  We make some observations which will be proved to be
valid in general in the proposition to follow.\\ Observe that, the
point $(-1,1,1) \in Y_1^1(\Z)$ gives
$$[7]^+(-1,1,1)=(1,3,1)~,~[7]^-(-1,1,1) = (-5,1,1) \in Y_7^1(\Z),$$
moreover $$[7]_\ast \circ [7]^+(-1,1,1) = [7]_\ast \circ [7]^- (-1,1,1)=(-1,1,1),$$
but $$[7]^+ \circ [7]_\ast(1,3,1) = (1,3,1), [7]^- \circ [7]_\ast(1,3,1) = (-5,1,1) \neq (1,3,1).$$
The point $(20,18,7) \in Y_1^1(\Z)$ gives
$$[7]^+(20,18,7) = (94,16,7) \in Y_7^1(\Z)~,~[7]^-(20,18,7) =
(-14,56,7) \in Y_7(\Z) \setminus Y_7^1(\Z)$$
and we have
$$[7]^+ \circ [7]_\ast(94,16,7)=(94,16,7), [7]^- \circ [7]_\ast(94,16,7)= (-14,56,7), [7]_\ast \circ [7]^+(20,18,7)  = (20,18,7).$$
These observations are quite general in nature and the general case
is proved in the proposition below. We remark that if $p^2 \mid D$, then $pD$ is not an admissible integer, hence we are not concerned about the image of a point $(x_0,y_0,z_0) \in Y^1_D(\Z)^*$ under level raising by $p$-map if $p^2 \mid D$.\\

\begin{proposition}\label{behaviourunderlevelmaps}
Let $D$ be an admissible integer, $p \equiv 1 \pmod{3}$ be a prime. Let $(x_0,y_0,z_0)$ be a point in $Y^1_D(\Z)^*$.
\begin{enumerate}
\item If $p \nmid Dz_0$, then both $[p]^\pm(x_0,y_0,z_0)
\in Y^1_{pD}(\Z)^*$. \\Otherwise, exactly one of the two points
$[p]^+(x_0,y_0,z_0),
[p]^-(x_0,y_0,z_0)$ belongs to $Y^1_{pD}(\Z)^*$ if $p^2 \nmid D$.\\
\item If $p \mid D$, then exactly one of $\{ [p]^+ \circ
[p]_\ast (x_0,y_0,z_0), [p]^- \circ [p]_\ast (x_0,y_0,z_0) \}$ is
$(x_0,y_0,z_0)$.\\
\item  Suppose $p^2 \nmid D$. If $[p]^\pm(x_0,y_0,z_0) \in Y^1_{pD}(\Z)^*$, then $[p]_\ast \circ [p]^\pm(x_0,y_0,z_0) = (x_0,y_0,z_0)$.
\end{enumerate}
\end{proposition}

\begin{proof} For convenience of notation, let us write $[p]^{\pm}(x_0,y_0,z_0)=(x_1^\pm, y_1^\pm,z_0)$ and $[p]_\ast(x_0,y_0,z_0) = (\tilde{x},\tilde{y},z_0)$.
 \begin{enumerate}
\item First, consider the case $p \nmid D$.
Suppose $\ell_1^2 \mid D$ (so $\ell_1 \neq p$ as
 we are in the case $p \nmid D$). Since $(x_0,y_0,z_0) \in Y^1_D(\Z)^*$, we have $\ell_1 \nmid y_0$.
We first show that $\ell_1 \nmid (y_1^+ y_1^-)$. Suppose $\ell_1
\mid y_1^\pm$, then $\ell_1$ divides $4pDz_0^3-3(y_1^\pm)^2
=(x_1^\pm)^2$; thus, $\ell_1 \mid x_1^\pm$. Hence $\ell_1$ divides
$uy_1^\pm \pm vx_1^\pm = py_0$, contradiction.

 \noindent Thus, if $p \nmid D$ and $p \nmid z_0$ , then it follows that both $(x_1^\pm, y_1^\pm,z_0) \in Y^1_{pD}(\Z)^*$. We remark
 that the condition $p \nmid z_0$ is essential as otherwise from equation \eqref{pdividesproduct} it follows that
 $p \mid y_1^+ y_1^-$ and hence $p \mid (y_1^+, z_0)$
 or $p \mid (y_1^-,z_0)$ which implies both $(x_1^\pm, y_1^\pm, z_0)$ can not be in $Y^1_{pD}(\Z)^*$.\\
\noindent Now suppose $p \mid Dz_0$ and $p^2 \nmid D$. Note that
\begin{equation}\label{pdividesproduct}
y_1^+ y_1^- =u^2y_0^2- v^2x_0^2 \equiv -v^2(x_0^2+3y_0^2) \equiv -4Dv^2z_0^3 \equiv 0 \pmod{p},
\end{equation}
hence $p$ divides at least one of $y_1^+$ and $y_1^-$.
\noindent If $p \nmid D$, but $p \mid z_0$, then $[p]^\pm(x_0,y_0,z_0) \in Y^1_{pD}(\Z)^*$ iff $p \nmid y_1^\pm$.
If $p \mid y_1^+$ and $p \mid y_1^-$, then $p$ divides $y_1^+ + y_1^- = 2uy_0$,
which is not possible as $p \nmid u$ and $(y_0,z_0)=1$.
Thus $p$ divides exactly one of $y_1^+$ and $y_1^-$.\\
\noindent If $p \mid \mid D$, then $p^2 \mid \mid pD$. Now from the
definition of $Y^1_{pD}(\Z)^*$ it follows that $[p]^\pm(x_0,y_0,z_0)
= (x_1^\pm, y_1^\pm, z_0) \in Y^1_{pD}(\Z)^*$ iff $p \nmid y_1^\pm$.
Now if $p \mid y_1^+$ and $p \mid y_1^-$, then $p$ divides both
$y_1^- + y_1^+=2uy_0$ and $y_1^- - y_1^+ = 2vx_0$. Thus, $p \mid
x_0$ and $p \mid y_0$, which implies that $p^2$  divides
$x_0^2+3y_0^2 = 4Dz_0^3$. This would mean that $p \mid z_0$ as $p^2
\nmid D$. Then $p \mid (y_0,z_0)$, contradiction. Hence $p$ divides
exactly one of $y_1^+$ and $y_1^-$.

\item As $p \mid D$, from equation \eqref{pdividesproduct} it follows that $p$ divides exactly one of
$y_1^-=uy_0+vx_0$ and $y_1^+=uy_0-vx_0$. If $p \mid y_1^-$, then $\tilde{x}= \frac{x_1^-}{p}$ and $\tilde{y}= \frac{y_1^-}{p}$. We obtain
$$[p]^+ \circ [p]_\ast (x_0,y_0,z_0) = [p]^+(\frac{x_1^-}{p}, \frac{y_1^-}{p},z_0) =(\frac{ux_1^-+3vy_1^-}{p}, \frac{uy_1^- - vx_1^-}{p},z_0) =  (x_0,y_0,z_0),$$
$$[p]^- \circ [p]_\ast (x_0,y_0,z_0) = [p]^-(\frac{x_1^-}{p}, \frac{y_1^-}{p},z_0) =(\frac{ux_1^--3vy_1^-}{p}, \frac{uy_1^- + vx_1^-}{p},z_0) =  (x_0-6v\tilde{y},y_0+2v\tilde{x},z_0).$$
Note that $\tilde{x}\tilde{y} \neq 0$ as $(\tilde{x},\tilde{y},z_0) \in Y^1_{\frac{D}{p}}(\Z)^*$, hence $[p]^- \circ [p]_\ast(x_0,y_0,z_0) \neq (x_0,y_0,z_0)$. The case $p \mid y_1^+$ is similar.
\item Suppose that $[p]^+(x_0,y_0,z_0) = (x_1^+,y_1^+,z_0) \in Y^1_{pD}(\Z)^*$. Note that $uy_1^+ + vx_1^+ =
py_0$ and $ux_1^+-3vy_1^+ = px_0$. As $p \mid (uy_1^+ + vx_1^+)$, we obtain
$[p]_{\ast}(x_1^+,y_1^+,z_0)=(x_0,y_0,z_0)$. The case $[p]^-(x_0,y_0,z_0) \in Y^1_{pD}(\Z)^*$ is similar.
\end{enumerate}
\end{proof}

\begin{lemma}\label{induc1}
Let $p$ and $q$ be two distinct prime congruent to 1 modulo $3$. Let $p=u_1^2+3v_1^2$  and $q= u_2^2 + 3v_2^2$. Let $u'=u_1u_2-3v_1v_2, v'=u_2v_1+u_1v_2$ and $u''=u_1u_2+3v_1v_2, v''=u_2v_1-u_1v_2$. We have,
$$[p]^+ \circ [p]^+(x_0,y_0,z_0):= [p]^{++}(x_0,y_0,z_0) = ((u_1^2-3v_1^2)x_0+6u_1v_1y_0, (u_1^2-3v_1^2)y_0-2u_1v_1x_0, z_0),$$
$$[p]^- \circ [p]^-(x_0,y_0,z_0):=[p]^{--}(x_0,y_0,z_0) = ((u_1^2-3v_1^2)x_0-6u_1v_1y_0, (u_1^2-3v_1^2)y_0+2u_1v_1x_0, z_0),$$
$$[p]^+ \circ [q]^+(x_0,y_0,z_0) = [q]^+ \circ [p]^+(x_0,y_0,z_0) =(u'x_0+3v'y_0, u'y_0-v'x_0, z_0),$$
$$[p]^+ \circ [q]^-(x_0,y_0,z_0) = [q]^- \circ [p]^+(x_0,y_0,z_0) =(u''x_0+3v''y_0, u''y_0-v''x_0, z_0),$$
$$[p]^- \circ [q]^+(x_0,y_0,z_0) = [q]^+ \circ [p]^-(x_0,y_0,z_0) =(u''x_0-3v''y_0, u''y_0+v''x_0, z_0),$$
$$[p]^- \circ [q]^-(x_0,y_0,z_0) = [q]^- \circ [p]^-(x_0,y_0,z_0) =(u'x_0-3v'y_0, u'y_0+v'x_0, z_0).$$
\end{lemma}

\begin{proof}
A straightforward calculation works.
\end{proof}

\begin{rem}\label{induc2}
Observe that $(u_1^2-3v_1^2, 6u_1v_1)=1$ and $(u_1^2-3v_1^2)^2+ 3(2u_1v_1)^2 =p^2$. This is the unique representation (up to sign) of $p^2$ as $\alpha^2+ 3 \beta^2$ with $(\alpha, 3 \beta)=1$. Similarly, $(u', 3v')=(u'',3v'')=1$ and $u'^2+ 3 v'^2 =u''^2+3v''^2=pq$. Also, these are the only representations
of $pq$ (up to sign) as $\alpha^2+ 3 \beta^2$ with $(\alpha, 3 \beta)=1$.\\
\noindent We also remark that for any $(x_0,y_0,z_0) \in Y^1_{D}(\Z)$ (with $p \nmid D$),  $[p]^+ \circ [p]^-(x_0,y_0,z_0) = [p]^- \circ [p]^+(x_0,y_0,z_0) =
(px_0, py_0, z_0) \notin Y^1_{p^2D}(\Z)^*$.
\end{rem}

\begin{lemma}
Let $D$ be an admissible integer and $p \equiv 1 \pmod{3}$ a prime such that $p \nmid D$. Let $(x_0,y_0,z_0)$ be a point in $Y^1_D(\Z)^*$.
Then, $[p]^{++}(x_0,y_0,z_0) \in Y^1_{p^2D}(\Z)^*$ iff
$[p]^{+}(x_0,y_0,z_0) \in Y^1_{pD}(\Z)^* $. Similarly,
$[p]^{--}(x_0,y_0,z_0) \in Y^1_{p^2D}(\Z)^*$ iff
$[p]^{-}(x_0,y_0,z_0) \in Y^1_{pD}(\Z)^*$.
\end{lemma}

\begin{proof}
 Let $(x_0,y_0,z_0) \in Y^1_{D}(\Z)^*$. For convenience of notation, let us write $[p]^{\pm}(x_0,y_0,z_0)=(x_1^\pm, y_1^\pm,z_0)$ and $[p]^{\pm \pm}(x_0,y_0,z_0)=(x_2^\pm, y_2^\pm,z_0)$.
 Suppose $\ell_1$ is a prime such that $\ell_1^2 \mid D$ (this necessarily mean $\ell_1 \neq p$). Then, $\ell_1 \nmid y_0$ by the definition of $Y^1_{D}(\Z)^*$. We first show that $\ell_1 \nmid y_2^+ y_2^-$. Suppose $\ell_1 \mid y_2^\pm$,
  then $\ell_1^2 \mid 4p^2Dz_0^3 - 3 (y_2^\pm)^2 = (x_2^\pm)^2$, hence $\ell_1 \mid x_2^\pm$. As a consequence, $ \ell_1$ divides $(u^2-3v^2)y_2^\pm \pm 2uv x_2^\pm = p^2y_0$, a contradiction.\\

\noindent Now note that
$$y_2^+ y_2^- = (u^2-3v^2)^2 y_0^2- (2uv)^2x_0^2 \equiv -(2uv)^2(x_0^2+3y_0^2) \equiv -4(2uv)^2Dz_0^3 \pmod{p}.$$
Thus if $p \nmid z_0$, then $p \nmid y_2^+ y_2^-$, hence both $[p]^{++}(x_0,y_0,z_0), [p]^{--}(x_0,y_0,z_0) \in Y^1_{p^2D}(\Z)^*$. Recall that (by Proposition \ref{behaviourunderlevelmaps}) in this case both $[p]^{\pm}(x_0,y_0,z_0) \in Y^1_{p^2D}(\Z)^*$.\\
Now suppose that $p \mid z_0$. Then $[p]^\pm(x_0,y_0,z_0) \in
Y^1_{pD}(\Z)^*$ iff $p \nmid y_1^\pm$, which is equivalent to $p
\nmid x_1^\pm$. Then
$$y_2^\pm = (u^2-3v^2)y_0 \mp 2uvx_0 \equiv \mp 2v(ux_0 \pm 3vy_0) = \mp 2vx_1^\pm \pmod{p}.$$
Thus if $p \mid z_0$, then $(x_2^\pm, y_2^\pm,z_0) \in
Y^1_{p^2D}(\Z)^*$ iff $p \nmid y_2^\pm$ which is equivalent to $p
\nmid x_1^\pm$ which, in turn, is equivalent to $(x_1^\pm, y_1^\pm,
z_0) \in Y^1_{pD}(\Z)^*$.
\end{proof}

\begin{lemma}
Let $D$ be an admissible integer with $3 \nmid D>1$. Then every
element of $Y_D^1(\Z)^*$ is the image of some element of
$Y_1^1(\Z)$. More precisely, let $R_D= \{(u_j, v_j) \in \Z^2 \mid
D=u_j^2+3v_j^2, (u_j, 3v_j)=1, u_j>0, v_j >0 \}$. Then every element
of $Y_D^1(\Z)^*$ is of the form $(u_jx_0 \pm 3v_jy_0, u_jy_0 \mp
v_jx_0, z_0)$ for some $(x_0,y_0,z_0) \in Y_1^1(\Z)$ and some
$(u_j,v_j) \in R_D$.
\end{lemma}

\begin{proof}
The lemma follows from the previous lemmata in this section but, in
order to make the proof more transparent, we give precise details
here. Note that $D$ is a cube-free integer $>1$ which is a product
of primes of the form $3k+1$. It suffices to prove the more precise
assertion:\\
{\it Claim.} Every element of $Y_D^1(\Z)^*$ is of the form $(ux_0
\pm 3vy_0, uy_0 \mp vx_0, z_0)$ for some $(x_0,y_0,z_0) \in
Y_1^1(\Z)$ and some $(u,v) \in R_D$. \\
To prove this claim, we apply induction on $\Omega(D)$,
the number of prime factors of $D$ counted with multiplicity.\\

\noindent If $\Omega(D)=1$, then $D=p$, a prime congruent to $1$ modulo $3$. So
    $$R_p = \{ (u, v) \mid p= u^2+3v^2, u>0, v>0, (u,3v)=1 \}.$$
Let $(x,y,z) \in Y_D^1(\Z)^*$. Then $xyz \neq 0, x^2+3y^2=4pz^3, (y,z)=1$.
As we already observed, $p=u^2+3v^2$ for unique positive integers $u,v$ such that $(u,3v)=1$. Hence, $(u,v) \in R_p$.
Further, as we have shown in the proof of Proposition 3.5, $p$ divides exactly one of the integers $uy-vx, uy+vx$. \\
If $uy \equiv vx$ mod $p$, then $uy + vx \not\equiv 0$ mod $p$, and $ux+3vy \equiv 0$ mod $p$. \\
If $uy \equiv -vx$ mod $p$, then $uy - vx \not\equiv 0$ mod $p$, and $ux-3vy \equiv 0$ mod $p$. \\
Thus, if $uy \equiv vx$ mod $p$, then $(x_1,y_1,z_1) := \bigg( \frac{ux+3vy}{p}, \frac{uy-vx}{p}, z \bigg) \in Y_1^1(\Z)$ and
$$[p]^-(x_1,y_1,z_1) = (ux_1-3vy_1, uy_1+vx_1, z) = (x,y,z).$$
Similarly, if $uy \equiv -vx$ mod $p$, then $(x_1,y_1,z_1) := \bigg( \frac{ux-3vy}{p}, \frac{uy+vx}{p}, z \bigg) \in Y_1^1(\Z)$ and
$$[p]^+(x_1,y_1,z_1) = (ux_1+3vy_1, uy_1-vx_1, z) = (x,y,z).$$
This proves the claim when $\Omega(D)=1$.

\noindent Now, let $\Omega(D) > 1$ and assume that the statement holds for admissible integers $A$ co-prime to $3$
for which $\Omega(A) < \Omega(D)$.
In other words, we assume for such $A$ that every element of
$Y_A^1(\Z)^*$ is of the form $(ux_0 \pm 3vy_0, uy_0 \mp
vx_0, z_0)$ for some $(x_0,y_0,z_0) \in Y_1^1(\Z)$ and some $(u,v) \in R_A$. \\
Write $D=Aq$ for a prime $q= u_1^2 + 3 v_1^2$ congruent to $1$
modulo $3$. There are two possibilities: either $q \mid A$ or $q
\nmid A$.\\

\noindent First, we assume that $q \nmid A$. Now, if $(u,v) \in
R_A$, then $u^2+3v^2=A, (u, 3v)=1, u>0, v>0\}.$
   Then $(|u^+|,|v^-|), (|u^-|,|v^+|) \in R_{Aq}$,
    where $u^\pm = u u_1 \pm 3v v_1, v^\pm = v u_1 \mp u v_1$. \\
Let $(x,y,z) \in Y_{Aq}^1(\Z)^{\ast}$. Then,
$$[q]_{\ast}(x,y,z) = (x_1,y_1,z) \in Y_A^1(\Z)^{\ast},$$
where $x_1 = \frac{u_1x \pm 3v_1y}{q}, y_1 = \frac{u_1y \mp v_1x}{q}$
and the signs are such that the entries are integers.\\
By induction hypothesis, any element of $Y_{A}^1(\Z)^*$ is of the
form $(u x_0 \pm 3vy_0, uy_0 \mp vx_0, z_0)$ for some $(x_0,y_0,z_0)
\in Y_1^1(\Z)$ and some $(u,v) \in R_A$. Therefore, $(x_1,y_1,z) =
(u x_0 \pm 3vy_0, uy_0 \mp vx_0, z_0)$ for some $(x_0,y_0,z_0) \in
Y_1^1(\Z)$ and some $(u,v) \in R_A$. Now $[q]^{\mp}(x_1,y_1,z_1) =
(x,y,z)$ where the signs are as in $x_1 = \frac{u_1x \pm 3v_1y}{q},
y_1 = \frac{u_1y \mp v_1x}{q}$.\\ In order to not confuse with the
sign appearing in $(x_1,y_1,z) = (u x_0 \pm 3vy_0, uy_0 \mp vx_0,
z_0)$, we consider the two cases separately: (i) when $x_1 =
\frac{u_1x + 3v_1y}{q}, y_1 = \frac{u_1y - v_1x}{q}$, and (ii)
when $x_1 = \frac{u_1x - 3v_1y}{q}, y_1 = \frac{u_1y + v_1x}{q}$.\\
In case (i), we have
$$x = u_1x_1 - 3v_1y_1 = u_1(u x_0 \pm 3vy_0) - 3v_1 (uy_0 \mp
vx_0)$$ $$= (u_1u \pm 3vv_1)x_0 + 3 (\pm vu_1 - v_1u) y_0 = u_2x_0 +
3v_2y_0$$
where $u_2 = u_1u \pm 3vv_1, v_2 = \pm vu_1 - v_1u$.\\
Also, in this case (i), we have
$$y = u_1y_1 + v_1x_1 = u_1(u y_0 \mp vx_0)+ v_1(u x_0 \pm 3vy_0)$$
$$= (\mp vu_1 + v_1u)x_0 + (u_1u \pm 3vv_1)y_0 = -v_2x_0 + u_2y_0.$$
As $(x_0,y_0,z) \in Y_1^1(\Z)$ implies $(\pm x_0, \pm y_0,z) \in
Y_1^1(\Z)$, we may take $u_3=|u_2|$ and $v_3=|v_2|$ such that
$(x,y,z) = (u_3x_0'+ v_3y_0', u_3y_0'-v_3x_0',z)$. This proves the
claim in case (i). The case (ii) is completely analogous.\\

\noindent Finally, we consider the second possibility $q \mid A$;
hence $D = Bq^2$ where $q \nmid B$ (as $D$ is cube-free) and $B$ is
an admissible integer not divisible by $3$. \\
Write $q= u_1^2+3v_1^2$ as before; we have $q^2 = \alpha^2+ 3
\beta^2$ where $\alpha = u_1^2-3v_1^2, \beta = 2u_1v_1$.\\
Let $(x,y,z) \in Y_{Bq^2}^1(\Z)^{\ast}$. So, $x^2+3y^2=4Bq^2z^3$,
and we have then
$$(\alpha y + \beta x)(\alpha y - \beta x) = \alpha^2 y^2 - \beta^2 x^2
\equiv (\alpha^2 + 3 \beta^2) \equiv 0~~mod~~q^2.$$ So $q$ divides
one of $\alpha y \pm \beta x$. If it divides both, then $ q
\mid y$ which is a contradiction to the fact $(x,y,z) \in Y_{Bq^2}^1(\Z)^{\ast}$
(as this implies $q \nmid y$ as $q^2 \mid D=Bq^2$). Hence $q$ divides exactly
one of $\alpha y \pm \beta x$.\\
Consider first the case when $q$ divides $\alpha y - \beta x$ and
does not divide $\alpha y + \beta x$. Then the fact that $q^2$
divides $\alpha^2y^2 - \beta^2x^2$ implies that $q^2$ divides
$\alpha y - \beta x$.\\
Again, the equality $(\alpha x + 3 \beta y)^2 + 3(\alpha y  -\beta
x)^2 = 4Bq^4$ gives that $q^2$ divides $\alpha x + 3 \beta y$.\\
Observe
$$(x_1,y_1,z) := \bigg( \frac{\alpha x + 3 \beta y}{q^2},
\frac{\alpha y - \beta x}{q^2}, z\bigg) \in Y_B^1(\Z)^{\ast}.$$ By
induction hypothesis, one can write
$$(x_1,y_1,z) = (ux_0 \pm 3v y_0, uy_0 \mp
vx_0,z)$$ where $(x_0,y_0,z) \in Y_1^1(\Z)$ and $(u,v) \in R_B$. We
have
$$[q]^{--}(x_1,y_1,z)= (\alpha x_1 - 3 \beta y_1, \alpha y_1 + \beta x_1,z)
= (x,y,z).$$ Putting $(x_1,y_1,z) = (ux_0 \pm 3v y_0, uy_0 \mp
vx_0,z)$, we have
$$(x,y,z) = (x_0(\alpha u \pm 3 \beta v) + 3y_0(\pm \alpha v - \beta
u), y_0(\alpha u \pm 3 \beta v) - x_0(\pm \alpha v - \beta u),z) =
u_2x_0 + 3v_2 y_0, u_2 y_0 - v_2 x_0,z)$$ where $(u_2=\alpha u \pm 3
\beta v, v_2 =\pm \alpha v - \beta u$. We change the signs of
$x_0,y_0$ to ensure that $u_2,v_2$ are positive. Note that then
$(u_2,v_2) \in R_{Bq^2}$.\\
The case when $q$ divides $\alpha y + \beta x$ and does not divide
$\alpha y - \beta x$ is completely analogous; we will use $[q]^{++}$
in that case.\\
Hence, the lemma is proved.
\end{proof}

\vskip 3mm

\begin{proposition}\label{X1D}
Let $D$ be an admissible integer with $3 \nmid D$. Suppose
$(x_0,y_0,z_0) \in Y^1_1(\Z)$ and $(u_j,v_j) \in R_D$ where $R_D$ is
as in the lemma above. Then
$$((u_jx_0  - 3v_jy_0) ,  ( u_jy_0 + v_jx_0), z_0)  \in Y_D^1(\Z)^*
\Leftrightarrow (D,z_0, u_j y_0 + v_jx_0)=1;$$
$$(  (u_jx_0  + 3v_jy_0) ,
 ( u_jy_0 - v_jx_0), z_0)  \in Y_D^1(\Z)^* \Leftrightarrow (D,z_0, u_j y_0
- v_jx_0)=1.$$
\end{proposition}

\begin{proof}
First, suppose $(  (u_jx_0  - 3v_jy_0) ,  ( u_jy_0 + v_jx_0), z_0)
\in Y_D^1(\Z)^*$, where $(x_0,y_0,z_0) \in Y^1_1(\Z)$.  Then, $(
u_jy_0 + v_jx_0 , z_0) = 1$ which evidently implies $(D,z_0, u_j y_0
+ v_jx_0)=1$. Conversely, suppose $(D,z_0, u_j y_0 + v_jx_0)=1$. We
will show that
$$((u_jx_0  - 3v_jy_0) , ( u_jy_0 + v_jx_0), z_0)
\in Y_D^1(\Z)^* \Leftrightarrow$$
$$(  (u_jx_0  - 3v_jy_0) ,  (
u_jy_0 + v_jx_0), z_0) \in Y_D^1(\Z).$$ Assume that $( (u_jx_0 -
3v_jy_0) , ( u_jy_0 + v_jx_0), z_0)  \in Y_D^1(\Z)$ but that $(
(u_jx_0  - 3v_jy_0) ,  ( u_jy_0 + v_jx_0), z_0)  \not\in
Y_D^1(\Z)^*$. Then, $(u_jy_0 + v_jx_0, z_0) > 1$ while $(D,z_0, u_j
y_0 + v_jx_0)=1$. Let $q$ be a prime dividing $(u_jy_0 + v_jx_0 ,
z_0)$ whereas $q \nmid D$. But, then
$$(u_j x_0 - 3v_j y_0)^2 + 3 (u_jy_0 + v_j x_0)^2 = 4Dz_0^3 \Rightarrow
q | (u_j x_0 - 3v_j y_0).$$
Hence, $q|(u_j(u_jx_0-3v_jy_0) +
3v_j(u_jy_0+v_jx_0))$; i.e, $q|Dx_0$. So, $q|x_0$ and hence
$q|u_jy_0$ as well as $q|3v_jy_0$. As $(u_j,3v_j)=1$, we get
$q|y_0$. This implies $q$ divides $x_0^2 + 3y_0^2 = 4z_0^3$ and
hence $q|z_0$ which is a contradiction to $(y_0,z_0)=1$. Therefore,
we have shown that $(  (u_jx_0  - 3v_jy_0) ,  ( u_jy_0 + v_jx_0),
z_0)  \in Y_D^1(\Z)^*$ iff $(  (u_jx_0  - 3v_jy_0) , ( u_jy_0 +
v_jx_0), z_0)  \in Y_D^1(\Z)$ iff $(D,z_0, u_j y_0 +
v_jx_0)=1$.\\
The other assertion is completely similar.
\end{proof}

\section{Irreducible trinomials up to rational equivalence.}

Firstly, we describe the sets $X^D_1(\Z)^*$. To do this, we
essentially keep track of the results proved in Sections 2 and 3.
\vskip 3mm

\begin{theorem}\label{mainthm}
Let $D>1$ be an admissible integer with $3 \nmid D$. let $R_D= \{(u_j, v_j) \in \Z^2 \mid D=u_j^2+3v_j^2, u_j>0, v_j >0, (u_j, 3v_j)=1\}$ and $r_D=|R_D|$. Then, the set $X_1^D(\Z)^* =
\bigcup_{j=1}^{r_D} X^{D(j)}$ where $X^{D(j)}$ is given by
$$
\{ (9D(u_jy \pm v_jx) , D(u_jx \mp 3v_jy)  , 3Dz) \mid (x,y,z) \in
Y^1_1(\Z) , (D,z, u_jy \pm v_jx) =1 \}  \sqcup$$ $$\{ (D(u_jx \mp
3v_jy), D \bigg(\frac{(u_jy \pm v_jx)}{3} \bigg) , Dz) \mid (x,y, z)
\in Y^1_1(\Z), 3 \mid (u_jy \pm v_jx), (D,z, u_jy \pm v_jx) =1 \}.$$

\noindent The set $X_1^{9D}(\Z)^* = \bigcup_{j=1}^{r_D} X^{D(j)}$
where $X^{D(j)}$ is given by
$$\{ (27D(u_jx \mp 3v_jy), 9D(u_jy \pm v_jx), 9Dz) \mid (x,y,z) \in Y^1_1(\Z) , 3 \nmid (u_jy \pm v_jx), (D,z, u_jy \pm v_jx) =1 \}.$$

\end{theorem}

\begin{proof}
Note that from Proposition \ref{X1D}, we have $Y^1_D(\Z)^* =
\bigcup_{j=1}^{r_D} \tilde{Y}_{D(j)}$ where
$$\tilde{Y}_{D(j)} = \{ ( u_j x \mp 3v_jy , u_jy \pm v_jx,z)  \mid (x,y, z) \in Y^1_1(\Z), (D,z, u_jy \pm v_jx) =1\}.$$
Now statement (2) follows from Corollary \ref{9D} via the map $\delta_D$.\\
For the statement (1),  we first get $X^1_D(\Z)^*$ from
$Y^1_D(\Z)^*$ by Lemma \ref{D}. We see that $X^1_D(\Z)^* =
\bigcup_{j=1}^{r_D} X_{D(j)}$, where $X_{D(j)}$ is given by
$$
\{ (9(u_jy \pm v_jx) , (u_jx \mp 3v_jy)  , 3z) \mid (x,y,z) \in
Y^1_1(\Z) , (D,z, u_jy \pm v_jx) =1 \}  \sqcup$$ $$\{ ((u_jx \mp
3v_jy),  \bigg(\frac{(u_jy \pm v_jx)}{3} \bigg) , z) \mid (x,y, z)
\in Y^1_1(\Z), 3 \mid (u_jy \pm v_jx), (D,z, u_jy \pm v_jx) =1 \}.$$

Finally, using the map $\theta_D$, we obtain $X^D_1(\Z)^* =
\bigcup_{j=1}^{r_D} X^{D(j)}$.

\end{proof}

\begin{rem}
We remark that starting from a point $(x_0,y_0,z_0) \in Y^1_1(\Z)$
it is not so easy to determine exactly how many points we get in
$X^1_D(\Z)^*$ using level raising maps (even in the case when $D=p$
a prime). \\
For instance, $(-1,1,1) \in Y_1^1(\Z)$ gives three points $(9,-5,3),
(27,1,3), (1,1,1) \in X_7^1(\Z)$, the point $(37,1,7) \in Y_1^1(\Z)$
gives only two points $(351,71,21), (71,13,7) \in X_7^1(\Z)$ and the
point $(20,18,7) \in Y_1^1(\Z)$ gives just one point $(144,94,21)
\in X_7^1(\Z)$. It is possible to show that if $3 \mid (u \pm v)$,
where $p=u^2+3v^2$, then one point in $Y^1_1(\Z)$ will give rise to
at most three points in $X^1_p(\Z)$. On the other hand if $3 \mid
v$, that is,  if $p$ is a prime expressible as $m^2 + 27 n^2$, then
from a point in $Y^1_1(\Z)$ we get either $1$ or $2$ or $4$ points
in $X^1_p(\Z)$. We remark that for a prime $p \equiv 1 \pmod{3}$,
the condition $3 \mid v$ is equivalent to $2$ being a cubic residue
modulo $p$(see \cite[Proposition 9.6.2]{IR}).

\end{rem}

\noindent One can write
$X^D_1(\Z)^*$ for general admissible $D$ in parametric form and
characterize those elements $(x_0,y_0,z_0) \in X^D_1(\Z)^*$ for
which the trinomial $X^3-z_0X+y_0$ is irreducible. Towards that, we
observe: \vskip 2mm

\begin{lemma}\label{irreducibility}
Let $f(X) = X^3-aX+b \in \Z[X]$ be a cubic polynomial whose
discriminant is a perfect square. If $(a,b)=d>1$ is cube-free and
for each prime $l$ such that $l^2|d$, we have $l^3 \nmid b$, then
$f(X)$ is irreducible.
\end{lemma}

\begin{proof}
If $d>1$ is square-free, then for any prime divisor $l$  of $d$,  we have $l^2 \nmid b$. If not, then this implies $l \mid \mid a$ and $l^2 \mid b$. Hence, $ l^3 \mid 4a^3-27b^2 =c^2$ and hence $l^2 \mid c$, which implies $l^4 \mid c^2+27b^2 = 4a^3$, which implies $l^2 \mid a$, contradiction. Thus $f(X)$ satisfies Eisenstein criterion for the prime $l$ and hence irreducible. \\

If $l^2|d$ for some prime $l$, then $a=l^2A,
b= l^2B$ with $(l,B)=1$. If
$$X^3-aX+b = X^3 - l^2Ax + l^2B = (X+r)(X^2+sX+t),$$
then $rt = l^2B$, $s = -r$ and $r^2 -t = l^2A$. Now $l$ divides $r$
or $t$; if it divides only one of them, we have a contradiction from
$r^2-t= l^2A$. Hence $l |r,  l |t$. So, $t = r^2-l^2A \equiv 0$ mod
$l^2$. But $rt= l^2B$ with $(l,B)=1$ implies $l^2 \nmid t$ which is
a contradiction. Therefore, $X^3-aX+b$ is irreducible.
\end{proof}

\begin{theorem}\label{generalthm}
For an admissible integer $D>1$, the irreducible trinomials (up to
rational equivalence) of the form $X^3-aX+b$ with $(a,b)=D$ and
discriminant perfect square are given by $X^3 - zX +y$, where $y$
and $z$ are integers such that $(x,y,z) \in X^D_1(\Z)^*$. Note that discriminant of the polynomial $X^3-zX+y$ is $x^2$.
\end{theorem}

\begin{proof}
For $(x,y,z) \in X^D_1(\Z)^*$, $X^3-zX+y \in \Z[X]$ has discriminant
perfect square and $(y,z)=D$.
On the other hand, suppose $X^3-aX+b$ has discriminant $c^2 = 4a^3-27b^2$ and $(a,b)=D$. Thus $(c,b,a) \in X^D_1(\Z)$. Since we are considering polynomials up to rational equivalence, we may assume that if for any prime $\ell$, $\ell^2 \mid a$ then $\ell^3 \nmid b$. This is equivalent to $(c,b,a) \in X^D_1(\Z)^*$.\\
Since $D>1$, irreducibility of $X^3-zX+y$ follows from Lemma \ref{irreducibility}.

\end{proof}

\begin{rem}
(i) When $(a,b)=1$, the above theorem is not valid as some of the
trinomials coming from $X^1_1(\Z)$ are reducible. The
irreducible ones are determined in the main theorem below. \\
(ii)
We do not claim that these polynomials are rationally inequivalent.
In fact the four points  $(\pm x, \pm y , z) \in X^D(\Z)^*$ generate
only one trinomial up to rational equivalence; so even though we
have listed all the trinomials, we have listed each multiple times -
once for each occurrence of a point in $X^D(\Z)^*$.
\end{rem}

\vskip 3mm

\noindent Using the parametrization of $Y^1_1(\Z)$ and combining
Theorems \ref{X91}, \ref{X11}, \ref{mainthm} and \ref{generalthm}, we may now write down all irreducible
trinomials $X^3-aX+b \in \Z[X]$ whose discriminant is a perfect
square (up to rational equivalence).

\begin{theorem}\label{mainthm2}
Up to rational equivalence, any irreducible trinomial whose
discriminant is a perfect square is $X^3-aX+b$ where $gcd(a,b)$ must
be $D$ or $9D$, where $3 \nmid D$ and each prime divisor of $D$ is
congruent to $1$ modulo $3$. To describe all of them, write
$D=u_j^2+3v_j^2$ with $u_j, v_j>0$ and $(u_j, 3v_j)=1$, $j \in \{1,
\dots, r_D \}$. Let $s$ and $t$ be co-prime integers with $3 \nmid
(s+t)$. Up to rational equivalence, the irreducible trinomials of
the form $X^3-aX+b \in \Z[X]$ whose discriminant is a perfect square
are given as follows:\\
(i) Polynomials with $(a,b)=1$ are given by
$$f_{(s,t)}(X)=X^3-3(s^2-st+t^2)X \pm ((s+t)^3-9st^2).$$
(ii) Polynomials with $(a,b)=9$ are given by
$$h_{(s,t)}(X)= X^3 - 9(s^2-st+t^2)X +9(s^3-3s^2t+t^3) .$$
(iii) Polynomials with $(a,b)=D$ are given by

\noindent (a)
 $f_{D, j, \pm ,s,t, 1}(X)= X^3 - aX + b$ where
$$a= 3D(s^2-st+t^2), b = D(u_j(s+t)(s-2t)(2s-t) \mp 9v_jst(s-t))$$
if $\big(D, s^2-st+t^2, 3u_jst(s-t) \pm v_j(s+t)(s-2t)(2s-t)
\big)=1$.\\
        (b) $f_{D,j, \pm, s,t,2}(X)=X^3 - aX+b$ where
$$a = 3D(s^2-st+t^2), b = D(u_j((s+t)^3-9st^2) \mp
3v_j(s^3-3s^2t+t^3))$$
        if $\big(D, s^2-st+t^2,  u_j(s^3-3s^2t+t^3) \pm v_j((s+t)^3-9st^2)
        \big)=1$.\\
        (c) $g_{D,j, \pm, s,t, 1}(X)=X^3-aX+b$ where
$$a = D(s^2-st+t^2), b= D(u_jst(s-t) \pm
\frac{v_j}{3}(s+t)(s-2t)(2s-t))$$ if $ 3 \mid v_j$, and
         $\big(D, s^2-st+t^2, 3u_jst(s-t) \pm v_j(s+t)(s-2t)(2s-t) \big) =1$. \\
        (d) $ g_{D,j, \pm ,s,t, 2}(X)= X^3-aX+b$, where
        $$a = D(s^2-st+t^2), b = \frac{D}{3}(u_j(s^3+t^3-3s^2t) \pm
        v_j((s+t)^3-9st^2))$$
if $3 \nmid v_j$ (which means $3 \mid u_j \pm v_j$) and $\big(D,
s^2-st+t^2, u_j(s^3+t^3-3s^2t) \pm v_j((s+t)^3-9st^2) \big)=1$. We choose $+$ (resp. $-$) sign if and only if $ 3 \mid u_j+v_j$ (resp. $3 \mid u_j-v_j$).\\

\vskip 2mm

\noindent (iv) Polynomials with $(a,b)=9D$ are given by

(a) $h_{9D,j, \pm,s,t, 1}(X) = X^3-aX+b$, where
$$a = 9D(s^2-st+t^2), b = 9D(3u_jst(s-t) \pm v_j(s+t)(2s-t)(s-2t))$$
if $3 \nmid v_j$, and $\big( D, s^2-st+t^2, 3u_jst(s-t) \pm
v_j(s+t)(2s-t)(s-2t) \big)=1$.\\
(b) $h_{9D,j,
\pm, s,t, 2}(X) = X^3-aX+b$, where
$$a = 9D(s^2-st+t^2), b = 9D(u_j(s^3-3s^2t+t^3) \pm v_j
((s+t)^3-9st^2))$$ if $3 \nmid (u_j \pm v_j)$, and $\big( D,
s^2-st+t^2, u_j(s^3-3s^2t+t^3) \pm v_j ((s+t)^3-9st^2) \big)=1$.\\

\end{theorem}

\begin{proof}
The theorem is a consequence of Theorem \ref{X11}, Theorem \ref{X91}
and Theorem \ref{generalthm} as we explicitly write down points
$(x,y,z) \in X^D(\Z)^*$ using the parametrization of $Y^1_1(\Z)$ as
given in Theorem 2.8 (\cite{Cohen2}[Proposition 14.2.1(2)]) and
Theorem \ref{mainthm}.

\noindent (i) If $(a,b)=1$, then from Theorem \ref{X11}, we see the corresponding trinomials are $X^3-(s^2-st+t^2)X + st(s-t)$, $X^3-3 (s^2-st+t^2)X + (s+t)(2s-t)(s-2t)$ and
$X^3-3(s^2-st+t^2)X \pm ((s+t)^3-9st^2)$. Note that  $X^3-(s^2-st+t^2)X + st(s-t)$
has a root $t$ and hence is never irreducible. Also, $X^3-3 (s^2-st+t^2)X + (s+t)(2s-t)(s-2t)$
has a root $(s+t)$ and hence is never irreducible.

Note that $(s+t)^3-9st^2$ and $st(s-t)$ are always odd (as both $s$ and $t$ can not be even), hence
$$X^3-3 st(s-t)X \pm ((s+t)^3-9st^2) \equiv X^3+X+1 \pmod{2},$$
as a consequence we see that $X^3-3 st(s-t)X + (s+t)^3-9st^2$ is irreducible.

\vskip 2mm

\noindent (ii) The polynomials are obtained from Theorem \ref{X91}. The statement regarding irreducible polynomials follows
from Theorem \ref{generalthm}.

\vskip 2mm

\noindent (iii) Combining the parametrization of $Y^1_1(\Z)$ as
given in \cite{Cohen2}[Proposition 14.2.1(2)] and Theorem
\ref{mainthm} we get $X_1^D(\Z)^* = \cup_{j=1}^{r_D} X^{D(j)}$, where
$X^{D(j)}$ is the set given explicitly as

$$
 \big\{ (9D(3u_jst(s-t) \pm v_j(s+t)(s-2t)(2s-t)) , D(u_j(s+t)(s-2t)(2s-t) \mp 9v_jst(s-t))  , 3D(s^2-st+t^2))
 \mid$$
$$  \big(D, s^2-st+t^2, 3u_jst(s-t) \pm v_j(s+t)(s-2t)(2s-t) \big)=1 \big\}
\cup $$
$$\big\{ (\pm 9D(u_j(s^3-3s^2t+t^3) \pm v_j((s+t)^3-9st^2)) ,\pm D(u_j((s+t)^3-9st^2) \mp 3v_j(s^3-3s^2t+t^3)) , 3D(s^2-st+t^2))
\mid $$
  $$\big(D, s^2-st+t^2,  u_j(s^3-3s^2t+t^3) \pm v_j((s+t)^3-9st^2) \big)=1 \big\} \cup
  $$
  $$\big \{ (D(u_j(s+t)(s-2t)(2s-t)  \mp 9v_jst(s-t) ), D(u_jst(s-t) \pm \frac{v_j}{3}(s+t)(s-2t)(2s-t)) , D(s^2-st+t^2) ) \mid
  $$
$$3 \mid v_j, \big(D, s^2-st+t^2, 3u_jst(s-t) \pm v_j(s+t)(s-2t)(2s-t) \big) =1 \big \}
\cup $$
 $$\big \{  \pm (D(u_j((s+t)^3-9st^2)  \mp 3v_j(s^3+t^3-3s^2t)), \pm \frac{D}{3}(u_j(s^3+t^3-3s^2t) \pm v_j((s+t)^3-9st^2) ) , D(s^2-st+t^2) ) \mid
 $$
 $$3 \mid u_j \pm v_j, \big(D, s^2-st+t^2, u_j(s^3+t^3-3s^2t) \pm v_j((s+t)^3-9st^2) \big)=1 \big \}.
$$

\vskip 2mm

\noindent We remark that $$ 3 \mid 3u_jst(s-t) \pm
v_j(s+t)(s-2t)(2s-t) \Leftrightarrow 3 \mid v_j$$ and
$$3 \mid
u_j(s^3+t^3-3s^2t) \pm v_j((s+t)^3-9st^2) \Leftrightarrow 3 \mid
(u_j \pm v_j).$$ Since $(u_j, 3v_j)=1$, we get $3 \mid (u_j \pm
v_j)$ if and only if $3 \nmid v_j$. \\Also note that the points
$(\pm x, \pm y, z) \in X_1^D(\Z)^*$ give rise to only one trinomial up
to rational equivalence; we only consider the expressions for
$x,y,z$ while writing down the polynomials $f_{D,j,\pm,s,t,2}(X)$ and
$g_{D,j,\pm,s,t,2}(X)$. The statement regarding irreducibility of the
trinomials follows from Theorem \ref{generalthm}.

\vskip 2mm

\noindent (iv) Combining the parametrization of $Y^1_1(\Z)$ as given
in \cite{Cohen2}[Proposition 14.2.1(2)] and Theorem \ref{mainthm} we
get $X_1^{9D}(\Z)^* = \cup_{j=1}^{r_D} X^{9D(j)}$, where $X^{9D(j)}$
is the set

$$
 \big\{ (27D(u_j(s+t)(2s-t)(s-2t) \mp 9v_jst(s-t)) , 9D(3u_jst(s-t) \pm v_j(s+t)(2s-t)(s-2t))  , 9D(s^2-st+t^2)) \mid
 $$
 $$3 \nmid v_j \big( D, s^2-st+t^2, 3u_jst(s-t) \pm v_j(s+t)(2s-t)(s-2t) \big)=1 \big\} \cup $$
 $$\big\{ ( \pm 27D(u_j((s+t)^3-9st^2) \mp 3v_j(s^3-3s^2t+t^3)) , \pm 9D(u_j(s^3-3s^2t+t^3) \pm v_j ((s+t)^3-9st^2))  , 9D(s^2-st+t^2))
 \mid$$
 $$3 \nmid (u_j \pm v_j),  \big( D, s^2-st+t^2, u_j(s^3-3s^2t+t^3) \pm v_j ((s+t)^3-9st^2) \big)=1
 \big\}.$$

\vskip 2mm

\noindent We remark that $$3 \nmid (3u_jst(s-t) \pm
v_j(s+t)(2s-t)(s-2t)) \Leftrightarrow 3 \nmid v_j$$ and
$$3 \nmid
(u_j(s^3-3s^2t+t^3) \pm v_j ((s+t)^3-9st^2)) \Leftrightarrow 3 \nmid
(u_j \pm v_j).$$ Also note that the points $(\pm x, \pm y, z) \in
X_1^{9D}(\Z)^*$ give rise to only one trinomial up to rational
equivalence; we only consider the expressions for $x,y,z$ while
writing down the polynomials $h_{D,j,\pm,s,t,2}(X)$. The statement
regarding irreducible trinomials follows from Theorem
\ref{generalthm}.

\end{proof}
\vskip 5mm

\section{Cube-free natural numbers expressible as sums of two rational cubes}

As an accidental byproduct of our results above, we can partially
solve a classical problem.  A
classical, open problem (see \cite{Sel})  in number theory asks for
a classification of all cube free natural numbers which can be expressed as sums
of cubes of two rational  numbers.\\
We give an alternate description of these numbers in terms of
non-trivial integral points of $X_1(\Z)$. Let $n$ be a cube-free natural number. Observe that the affine
curve $x^3+y^3=n$ is isomorphic to the affine curve $X^2=4Z^3-27n^2$
given by the following change of variables:
\begin{equation}\label{xy}
x = \frac{9n+X}{6Z} \text{  and  } y = \frac{9n-X}{6Z},
\end{equation}
and
\begin{equation}\label{XY}
X= \frac{3n}{x+y} \text{  and } Y = 9n \frac{x-y}{x+y}.
\end{equation}
Let us denote by $E_n$ the elliptic curve whose Weierstrass equation
is given by $Y^2=4X^3-27n^2$. (We remark that the elliptic curve
$E_n$ is isomorphic to $E_{nm^3}$ over $\Q$.) Then we can identify
(the projectivization of) $x^3+y^3=n$ with the elliptic curve $E_n$.
As a consequence, we see that $n$ can be written as sum of two
rational cubes iff $E_n(\Q)$ is non-trivial. It is well known (see
\cite{DV}) that $E_n(\Q)_{\tors} = \{ \OO \}$ if $n>2$. Thus, a cube
free natural number $n>2$ is can be expressed as a sum of two
rational cubes iff $\mathrm{rk}(E_n(\Q)) >0$. The standard approach to
study this problem is via the theory of (mock) Heegner points. But,
in what follows, we relate this to the integral solutions of the
equation $X^2+27Y^2=4Z^3$.

\begin{defn}\label{niceset}
Let $S$ denote the set of cube-free natural numbers given by
$$S= \{ n \mid n>2,  (a, nm^3, b) \in X_1^D(\Z)^* \text{ for some } a,b, m \in \Z \text{ and for some admissible } D \}.$$
\end{defn}

\begin{theorem}\label{sumofcubes}
Let $n>2$ be a cube-free natural number. Then $n$ is a sum of cubes
of two rational numbers iff $n \in S$, where $S$ is as defined
above.
\end{theorem}

\begin{proof}

Let $n \in S$. Then there exist integers $a,b,m$ such that
$(a,nm^3,b) \in X_1^{non-triv}(\Z)$ - or, equivalently - $(a,b)$
satisfies the equation $X^2=4Z^3 - 27(nm^3)^2$. Then from
\eqref{xy}, it follows that $\Big( \frac{9nm^3+a}{6b},
\frac{9nm^3-a}{6b} \Big)$ satisfies the equation $x^3+y^3=nm^3$. As
a consequence, we have
$$n = \Big( \frac{9nm^3+a}{6bm} \Big)^3 +  \Big( \frac{9nm^3-a}{6bm} \Big)^3.$$\\

Conversely, let $n$ be a cube-free natural number and suppose $q_1 =
\frac{a_1}{db_1}$ and $q_2 = \frac{a_2}{db_2}$ (with $(a_1, db_1) =
(a_2, db_2) = (b_1, b_2)=1$) are two rational numbers such that
$q_1^3 + q_2^3 = n$. Then
$$(a_1b_2)^3+(a_2b_1)^3 = nd^3b_1^3b_2^3.$$
If $\ell \mid b_1$, then $\ell^3 \mid b_1^3$, which implies $\ell^3 \mid (nd^3b_2^3-a_2^3)b_1^3 = (a_1b_2)^3$ or equivalently $\ell \mid a_1b_2$. This is impossible as $(a_1, b_1) = (b_1, b_2)=1$. A similar argument shows that $\ell \mid b_2$ is also impossible. We conclude that $b_1b_2 = \pm 1$. If necessary, changing the signs of $a_1$ and $a_2$, we may assume that $b_1=b_2=1$. \\
Since $\big(\frac{a_1}{d}, \frac{a_2}{d} \big)$ satisfies $x^3+y^3 =n$, from \eqref{XY}, we see that $\big( 9n \frac{a_1-a_2}{a_1+a_2}, \frac{3nd}{a_1+a_2} \big)$ (note that $a_1+a_2 \neq 0$) satisfies $X^2=4Z^3-27n^2$. Hence
$$\Big( 9n \frac{a_1-a_2}{a_1+a_2},n, \frac{3nd}{a_1+a_2} \Big) \in X_1(\Q).$$
As a consequence, we see that  $\big( 9n
\frac{a_1-a_2}{a_1+a_2}m^3,nm^3, \frac{3nd}{a_1+a_2}m^2 \big) \in
X_1(\Q)$, for any natural number $m$. Since $a_1^3+a_2^3 = nd^3$, it
follows that $(a_1+a_2) \mid nd^3$. Taking $m=d$, we obtain
$$\Big( 9(a_1-a_2) \frac{nd^3}{a_1+a_2},nd^3, 3 \frac{nd^3}{a_1+a_2} \Big) \in X_1(\Z).$$
Now note that $a_1 \neq a_2$ as $n \neq 2$. Thus $\Big( 9(a_1-a_2)
\frac{nd^3}{a_1+a_2},nd^3, 3 \frac{nd^3}{a_1+a_2} \Big) \in
X_1^{non-triv}(\Z)$, here $X_1^{non-triv}(\Z) = \{ (x, y, z) \in \Z^3 \mid x^2+ 27y^2 =4z^3, xyz \neq 0 \}$, which are the complements of the trivial integrals zeros $X_1^{\text{triv}}(\Z)$. Observe that if $(x,y,z) \in X_1^{non-triv}(\Z)$
and $\ell^3 \mid y$ and $\ell^2 \mid z$, then $\ell^3 \mid x$ and
$\big( \frac{x}{\ell^3}, \frac{y}{\ell^3}, \frac{z}{\ell^2} \big)
\in X_1^{non-triv}(\Z)$. As a consequence we see that if $\ell \mid d$ is a prime such that
$\ell^2 \mid 3 \frac{nd^3}{a_1+a_2}$, then $\ell^3 \mid 9(a_1-a_2)
\frac{nd^3}{(a_1+a_2)}$ and hence $\Big( 9(a_1-a_2)
\frac{nd^3}{\ell^3(a_1+a_2)},\frac{nd^3}{\ell^3}, 3
\frac{nd^3}{\ell^2(a_1+a_2)} \Big) \in X_1^{non-triv}(\Z)$.
Proceeding in this manner, we see that there exists a largest
natural number $k$ such that $k \mid d$ and
$$\Big( 9(a_1-a_2) \frac{nd^3}{k^3(a_1+a_2)},\frac{nd^3}{k^3}, 3 \frac{nd^3}{k^2(a_1+a_2)} \Big):=(a,nm^3,b) \in X_1^{non-triv}(\Z).$$
Note that $(a,nm^3,b) \in X_1^D(\Z)$ for $D= gcd(nm^3,b)$. If for a
prime $\ell$, $\ell^3 \mid D$, then we see that $\Big( 9(a_1-a_2)
\frac{nd^3}{(k\ell)^3(a_1+a_2)},\frac{nd^3}{(k\ell)^3}, 3
\frac{nd^3}{(k\ell)^2(a_1+a_2)} \Big) \in X_1^{non-triv}(\Z)$, which
contradicts the maximality of $k$. Hence $D$ is cube free.\\
Similarly, if $\ell^3 \mid nm^3$ and $\ell^2 \mid b$, then again
$\Big( 9(a_1-a_2)
\frac{nd^3}{(k\ell)^3(a_1+a_2)},\frac{nd^3}{(k\ell)^3}, 3
\frac{nd^3}{(k\ell)^2(a_1+a_2)} \Big) \in X_1^{non-triv}(\Z)$, which
contradicts the maximality of $k$. Thus $(a,nm^3,b) \in X_1^D(\Z)^*$
for an admissible $D$.
\end{proof}
\vskip 3mm

\noindent In order to deduce that certain primes are expressible as
sums of two rational cubes, we recall that $(9(s^3-3s^2t+t^3),
(s+t)^3-9st^2 , 3(s^2-st+t^2) ) \in X^1_1(\Z)$ for any co-prime
integers $s,t$ with $3 \nmid (s+t)$. As a consequence, we obtain:

\begin{corollary}\label{pm1mod9assumofcubes}
Let $p \equiv \pm 1 \pmod{9}$ be a prime and $m$ be an integer such
that $pm^3=(s+t)^3-9st^2$ for some co-prime integers $s$ and $t$. Then $p$ can be written as a sum of two
rational cubes. Explicitly,
$$p= \Big( \frac{s^3+t^3-3st^2}{m(s^2-st+t^2)} \Big)^3 +  \Big( \frac{3st(s-t)}{m(s^2-st+t^2)} \Big)^3.$$
\end{corollary}

\begin{rem}
There are infinitely many primes which satisfy the hypotheses of
corollary \ref{pm1mod9assumofcubes}; see \cite[Theorem 1.1]{HBM}. By
work of Stag\'e \cite{Sta}, we know that odd primes $p$ which are
congruent to $2,5$ modulo $9$ cannot be expressed as a sum of two
rational cubes. Recent work of Dasgupta and Voight \cite{DV} shows
that primes $p$ congruent to $4$ or $7$ modulo $9$ can be expressed
as a sum of rational cubes if $3$ is not a cubic residue modulo $p$.
Rodriguez-Villegas and Zagier \cite{RVZ}  - assuming the truth of
the BSD conjecture - give a criterion to decide whether a prime $1$
modulo $9$ can be expressed as a sum of two rational cubes. The
authors do not know of any previous unconditional result describing
when primes congruent to $\pm 1$ modulo $9$ are expressible as sums
of two rational cubes.
\end{rem}
\begin{rem}
There are $49$ primes less than $2000$ which are congruent to $1$
modulo $9$ and, by the result of  \cite{RVZ}, we expect that $22$ of
them are expressible as sums of two rational cubes. If we vary
$(s,t,m) \in \Z^3$ with $0 < |s|, |t|, m  < 1000$, $(s,t)=1$, then all of these $22$ primes satisfy
$pm^3=(s+t)^3-9st^2$ for some choice of $(s,t,m)$. There are
$14$ primes less than $500$ which are congruent to $-1$ modulo $9$
and all of  these $14$ primes satisfy
$pm^3=(s+t)^3-9st^2$ for some choice of $(s,t,m)$ with $(s,t,m) \in \Z^3$ with $0 < |s|, |t|, m  < 1000$.
There are $50$ primes less than $2000$ which are congruent to $-1$
modulo $9$, at least $34$ out of these $50$ primes satisfy
$pm^3=(s+t)^3-9st^2$ for some choice of $(s,t,m)$.
All of these were verified using SAGE \cite{SAGE}.

From the above data, we may hazard a guess that there are infinitely
many primes congruent to $1$ (resp. $-1$) modulo $9$ which satisfy
the condition of Corollary \ref{pm1mod9assumofcubes}.

Further, one can verify that if a tuple $(a, nm^3,b) \in
X^D_1(\Z)^*$, then $(D,m)=1$. Since elements of $X^D_1(\Z)^*$ are of the form $(Dx,Dy,Dz)$, it follows that $D \mid nm^3$ and hence
$D \mid n$. Now, for a prime $p$ which is congruent to $-1$ modulo
$9$, if $(a, nm^3,b) \in X^D_1(\Z)^*$ then it follows that $D=1$, as
$X^p_1(\Z)^*=\emptyset$. As a consequence $p \equiv -1 \pmod{9}$ can
be written as sum of two rational cubes if and only if there exists
integers $s,t,m $ with $(s,t)=1$ for which at least one of the
following condition holds:
\begin{enumerate}
\item $pm^3 = (s+t)^3-9st^2$,
\item $pm^3= st(s-t).$
\end{enumerate}
Note that as $(s+t)(2s-t)(s-2t)= 2(s+t)^3 - 9st(s+t) \equiv \pm 2
\pmod{9}$, we did not include it in the list. We could not find any
tuple $(s,t,m) \in \Z^3$ such that $pm^3= st(s-t)$. This leads us to
believe that $p \equiv -1 \pmod{9}$ is expressible as sum of
rational cubes iff it satisfies the condition of Corollary
\ref{pm1mod9assumofcubes}.
\end{rem}
\vskip 5mm

\noindent {\bf Acknowledgments.}\\
It is a pleasure to thank the referee for her/his suggestions which
helped us improve the presentation of the results. The referee's
detailed comments also helped us in correcting some minor errors and
the language also at various places. We are indebted to the referee
for also pointing out the primes $883$ and $937$ which correspond to
the triples $(s,t,m) = (14, 195, 17)$ and $(s,t,m) = (8, 203, 19)$
in Remark 5.5. Finally, we also thank S. Panda and P. Shingavekar
for pointing out that the primes $701$ and $1151$ which correspond
to $(77,1299,127)$ and $(1137,-250,37)$ respectively. 
\vskip 5mm

\noindent {\bf Dipramit Majumdar}, Department of Mathematics, Indian
Institute of Technology Madras,
Chennai 600036, India (dipramit@gmail.com) \\

\noindent {\bf B. Sury}, Stat-Math Unit, Indian Statistical
Institute, 8th Mile Mysore Road, Bangalore 560059, India
(surybang@gmail.com)

\end{document}